\documentclass[review]{elsarticle}
\usepackage{setspace}
\usepackage{hyperref}
\usepackage{amsmath}
\usepackage{amssymb}
\usepackage[mathlines]{lineno}
\usepackage{amsthm, mathtools, empheq}
\usepackage{paralist, mathtools}
\usepackage{caption}
\usepackage{subcaption}
\usepackage{graphics}
\usepackage{graphicx}
\usepackage{xcolor}
\usepackage{ulem}
\usepackage[title]{appendix}
\usepackage{soul} 
\usepackage[english]{babel}
\usepackage[autostyle]{csquotes}
\usepackage{placeins}
\usepackage{enumerate, comment}
\usepackage{mathrsfs}
\usepackage{mathtools}
\usepackage{epstopdf}
\allowdisplaybreaks

\newtheorem{theorem}{Theorem}[section]
\newtheorem{corollary}{Corollary}

\newtheorem{lemma}[theorem]{Lemma}

\theoremstyle{definition}
\newtheorem{definition}[theorem]{Definition}
\newtheorem{remark}{Remark}

\newcommand{\ve}{\varepsilon}

\newcommand{\eff}{{\mathrm{eff}}}

\journal{}

\begin{document}

\begin{frontmatter}

\title{Homogenization of an indefinite spectral problem arising in population genetics}
\author[ad1]{Srinivasan Aiyappan}
\author[ad1]{Aditi Chattaraj}
\author[ad2]{Irina Pettersson\corref{mycorrespondingauthor}}
\cortext[mycorrespondingauthor]{Corresponding author}
\ead{irinap@chalmers.se}

\address[ad1]{Department of Mathematics, Indian Institute of Technology Hyderabad, Sangareddy, Telangana, India 502284}
\address[ad2]{Chalmers University of Technology and University of Gothenburg, Sweden}

\begin{abstract}
We study an indefinite spectral problem for a second-order self-adjoint elliptic operator in an asymptotically thin cylinder. The operator coefficients and the spectral density function are assumed to be locally periodic in the axial direction of the cylinder. The key assumption is that the spectral density function changes sign, which leads to infinitely many both positive and negative eigenvalues. The asymptotic behavior of the spectrum, as the thickness of the rod tends to zero, depends essentially on the sign of the average of the density function. We study the positive part of the spectrum in a specific case when the local average is negative. We derive a one-dimensional effective spectral problem that is a harmonic oscillator on the real line, and prove the convergence of spectrum. A key auxiliary result is the existence of a positive principal eigenvalue of an indefinite spectral problem with the Neumann boundary condition on a periodicity cell. This study is motivated by applications in population genetics where spectral problems with sign-changing weight naturally appear.
\end{abstract}

\begin{keyword}
Indefinite spectral problem, sign-changing spectral density, homogenization, thin domain, localization of eigenfunctions.

\vspace{2mm}
\MSC[2010] Primary: 35B27, 35B40, 35P15, 74K10, 35J25

\end{keyword}

\end{frontmatter}



\tableofcontents
\nolinenumbers

\section{Introduction}

This paper deals with the homogenization of an indefinite spectral problem for a second-order elliptic operator defined in a thin cylindrical domain. The key assumption is the presence of a sign-changing weight function in front of the spectral parameter and the local periodicity of the operator coefficients. We impose the homogeneous Neumann boundary condition on the lateral boundary of the cylinder and the homogeneous Dirichlet condition on the bases. We study the asymptotic behavior of the spectrum as the thickness of the cylindrical domain vanishes. 

Spectral problems with indefinite weight arise when modeling population genetics \cite{fleming1975selection,Yuan_Lou_Nagylaki_WeiMingNi-2013DCDS, Nagylaki-Biomath21}. Consider, for example, a model with two types of alleles $A_1~\text{and}~A_2$, and let $p(t,x)$ denote the frequency of the first allele $A_1$ at time $t$ and point $x$ in some bounded habitat $\mathcal{O}$. Then $p=p(t,x)$ is assumed to satisfy a nonlinear evolution equation:
\begin{align*}
        &\displaystyle\frac{\partial p}{\partial t} = m \Delta p + \lambda s(x) p(1-p) ~~ \mbox{in}~ (0, \infty)\times\mathcal{O},\\
        &  \nabla p  \cdot \nu = 0 ~~ \mbox{on} ~ (0,\infty)\times\partial \mathcal{O}.
\end{align*}
The term $m\Delta p$ represents the population dispersal, $m$ is the rate of migration or dispersal, $s(x)p(1-p)$ represents the effect of natural selection, and $\lambda>0$ is a parameter (appearing after non-dimensionalization of the problem). In particular, $s(x)$ is a selection coefficient favoring allele $A_1$ at location $x$ and $p(1-p)$ is an approximation for the genotype frequency in the case of weak selection. The unknown frequency satisfies $0\le p(x,t) \le 1$ in $\overline{\mathcal{O}}$. To model that a selective advantage at some regions of $\mathcal{O}$ might become a disadvantage in others, we assume that $s(x)$ changes sign. The homogeneous Neumann boundary condition can be interpreted such as there is no flow of genes into or out of $\mathcal{O}$. There are two trivial stationary solutions: $p=1$ (only the  first allele occurs in the population) and $p=0$ (only the second allele occurs). Bifurcation points of the above stationary solutions provide the existence of physically interesting non-constant equilibrium solutions. It has been proved in \cite{fleming1975selection} that if $\lambda>0$ is a bifurcation point of $p=0$, then $\lambda$ is an eigenvalue of the linearized problem 
\begin{equation}
\label{eq:linearized}
\begin{aligned}
-&m \Delta u(x) = \lambda s(x) u(x) ~~ \mbox{in}~ \mathcal O,\\
&  \nabla u  \cdot \nu = 0 ~~ \mbox{on} ~ \partial \mathcal{O},
\end{aligned}
\end{equation}
which is a spectral problem with a sign-changing weight. The stability of the trivial equilibria are determined by the sign of the average of the selection coefficient $\int_{\mathcal O} s(x)\, dx$ and by the value of the smallest positive eigenvalue $\lambda_1$ of \eqref{eq:linearized}. {In particular, if $\int_{\mathcal O} s(x)\, dx<0$, that is when the type $A_2$ is advantageous}, then $p=1$ is unstable for any $\lambda>0$, while $p=0$ is stable for $\lambda<\lambda_1$ and unstable for $\lambda>\lambda_1$. It has been proved that a non-trivial equilibrium appears when $\lambda>\lambda_1$ (see Sec. 4, \cite{fleming1975selection}).

It is natural to extend the above model to the case when the diffusion $m$ and selection criteria $s$ are spatially heterogeneous and to allow for oscillations in $x$, which motivates us to consider an indefinite spectral problem with rapidly varying coefficients, with a small period of oscillations denoted by $\ve>0$. The thickness of the cylidrical domain is assumed to $O(\ve)$, of the same order as the period of oscillations. Modeling population genetics in random media is technically challenging, and periodic approximations are commonly employed. In this work, however, we adopt a more general framework by assuming local periodicity of the coefficients, rather than the classical global periodicity (see problem \eqref{orig-prob} and hypotheses (H1)--(H5)). From the perspective of population genetics, the asymptotic behavior of the first positive eigenvalue of \eqref{orig-prob} plays a crucial role, as it determines a critical threshold beyond which instability—and consequently, the emergence of a non-trivial equilibrium—occurs. We focus on the asymptotic behavior of the positive eigenvalues $\lambda_\ve$, as $\ve \to 0$, and show that it is governed by the positive principal eigenvalue of an auxiliary spectral problem \eqref{eq:cell-problem} and grows as $1/\ve^2$ (see Theorem \ref{th:main}). Moreover, we analyze the asymptotic behavior of the corresponding eigenfunction, which in the context of population genetics can be interpreted as the spatial distribution of the frequency of a certain gene. This eigenfunction exhibits oscillatory behavior and is modulated by a rescaled eigenfunction of a harmonic oscillator problem on $\mathbb R$ (see problem \eqref{eq:eff-prob}). 
The key components of the present work are: the presence of an indefinite weight, locally periodic coefficients, and, as a consequence, the localization of eigenfunctions corresponding to positive eigenvalues in the case of negative local average of the spectral weight. In addition, the thin domain leads to the dimension reduction problem.  

Homogenization of spectral problems has been extensively studied starting from 70s. Kozlov in \cite{S.M.Kozlov_1980} studied the averaging phenomena of random operators. Homogenization of elliptic spectral problem with Dirichlet data was studied by Kesavan for both periodic and non-periodic domains \cite{S.Kesavan_part-1_1979,S.Kesavan_Part-2_1979}. In classical cases, the convergence of spectra follows from the convergence of the solutions to boundary value problems. Singularly perturbed spectral problem show, however, a different asymptotic behavior. Below we describe some closely related results.

In \cite{capdeboscq1998homogenization} a periodic spectral problem for a second-order elliptic operator with a large drift term was studied. It was shown that the eigenfunctions deviate exponentially in $\ve^{-1}$ from
the solutions of an eigenvalue problem for an homogenized diffusion equation, and
the corresponding eigenvalues are shifted by a constant factor. The spectral problem with neutronic multigroup diffusion was studied in \cite{ALLAIRE_Capdeboscq_2000} and a drift in linear transport was studied by Bal in \cite{Bal2001}. In these works, under periodicity assumptions, a factorization with an eigenfunction of an auxiliary spectral problem is used to describe the oscillations.  In \cite{Allaire-Piatniski-Commun.PDE} the authors consider the homogenization of the spectral problem for
a singularly perturbed diffusion equation in a periodic
setting. Under the
hypothesis that the first cell eigenvalue admit a unique minimum in the domain with non-degenerate
quadratic behavior, it was proved that the k-th
original eigenfunction is asymptotically given
by the product of the first cell eigenfunction at the $\ve$ scale and the eigenfunction of an harmonic oscillator problem at the $\sqrt{\ve}$ scale. 

In \cite{NazarovGomezMartinez-JDE2021}, the authors have considered a Dirichlet spectral problem in a bounded domain with a smooth boundary surrounded by a thin band. The eigenfunctions satisfy two different spectral problems, one inside the fixed bounded domain and another inside the thin band. Under natural interface conditions, it is shown that a localization phenomenon occurs near the extrema of the curvature see \cite{kamotskiiNazarov-2000eigenfunctions, Nazarov2002spectr_thin_domain, HelfferKachmar-RobinLaplacian_var._curvature} for maxima and for minima \cite{BorisovFreitas2009, SolomyakFreidlander-2009}. In the problems for banded domains, as in \cite{NazarovGomezMartinez-JDE2021}, they occur for both maxima and minima. 

It turns out that introducing, in addition to singular perturbation, local periodicity in the coefficients or in the geometry of the domain leads not only to oscillating eigenfunctions, but also a localization of eigenfunctions.
A singularly perturbed operator in a locally periodic setting where a localization of eigenfunctions occurs has been considered in \cite{piat2013localization}. It deals with a spectral problem for a second-order elliptic operator in a periodically perforated bounded domain with a Fourier boundary condition, and the coefficient in the boundary operator is a function of slow
argument. In the subcritical case, the asymptotic expansion is non-standard being a series in powers of $1/4$, and the localization of eigefunctions occurs in the same scaling. Similar singularly perturbed spectral problems where the localization of eigenfunctions takes place are considered in \cite{PaPe-ApplicAnal-2014}, \cite{KlasPe-MMA-2017}, and \cite{Pettersson2022JMAA}. 
In \cite{PaPe-ApplicAnal-2014}, the authors considered an elliptic spectral problem with a large potential of order $\ve^{-\beta}$ and locally periodic coefficients stated on a thin rod with locally periodic perforation. 
In the subcritical case when $0 < \beta < 2$, under the assumption that the potential has a global minimum at zero, it is proved that the $j$th eigenfunction can be approximated by a scaled exponentially decaying eigenfunction of order $\ve^{1/4}$ which is constructed in terms of the $j$th eigenfunction of an one-dimensional harmonic oscillator operator. 

The work \cite{Pettersson2022JMAA} is devoted to a spectral problem of an elliptic operator with locally periodic coefficients in a thin cylinder with locally periodic lateral boundary and homogeneous Dirichlet conditions. The eigenvalues are shown to be of order $\ve^{-2}$ when the thickness of the cylinder $\ve$ tends to zero, and are described in terms of the first eigenvalue of the auxiliary spectral cell problem.

Another situation when localization of eigenfunctions appear is when studying the asymptotics of the indefinite spectral problems.
One of the pioneers in studying indefinite spectral problems was Hilbert \cite{hilbert1924grundzuge}, who proved the existence of two infinite sequences of eigenvalues, positive and negative. Later, spectral problems with sign-changing weight has attracted a lot of attention due to the applications in population genetics. In \cite{fleming1975selection} the author analyzed bifurcation and stability in a selection-migration model, showing that indefinite weight leads to stability switching and multiple equilibria depending on dispersal rate. The existence of principal eigenvalues of the Laplacian with Dirichlet and Neumann boundary conditions in a bounded domain can be found in \cite{brown1980existence}. The question of completeness of eigenfunctions corresponding to positive and negative eigenvalues have been studied by Pyatkov \cite{pyatkov1998indefinite}.

The asymptotics of a Dirichlet spectral problem for an elliptic operator with indefinite weight in a bounded domain with periodic coefficients is studied for all cases of the sign of the average in \cite{nazarov2011homogenization}. If the average of the weight is strictly positive, the positive eigenvalues and the
corresponding eigenfunctions show the same regular limiting behavior as in the case of pointwise positive spectral density. If the mean value of zero, then the limit spectral problem generates a quadratic operator pencil, and the eigenvalues are of order $1/\ve$. In the case of negative average, the positive eigenvalues are of order $1/\ve^2$, and the corresponding eigenfunctions prove to be rapidly oscillating. The technique in the later case involved describing the Bloch spectrum of a periodic operator.

In \cite{pankratova2011homogenization}, the authors derived the asymptotics for the positive eigenvalues in the case of positive local average of the spectral weight, and both positive and negative when the average changes sign. The most difficult case was not studied: the asymptotics of the negative part of the spectrum for the positive local average. This is the focus of the present paper. We will show (see Theorem \ref{th:main}) that the positive eigenvalues grow as $1/\ve^2$, and the corresponding eigenfunctions localize under the assumption that the principal positive eigenvalue of an auxiliary cell problem attains a unique global minimum in the domain. The existence of a positive principal eigenvalue of a cell spectral problem is of interest in its own. Moreover, we study the regularity of eigenvalues and eigenfunctions of the the cell spectral problem with respect to the slow variable, which plays a role of a parameter. 

We mention also \cite{piat2012steklov}, where the homogenization of Steklov spectral problem for a divergence form elliptic operator has been done in a periodically perforated domain under the assumption that the spectral weight function changes sign. It has been shown that in the case of negative average of the spectral weight function over the boundary of the holes, the positive eigenvalues grow as $1/\ve$,  and the corresponding eigenfunctions are rapidly oscillating. For the qualitative theory of nonlinear indefinite spectral problems we refer to \cite{hess1980some}, and for the homogenization of such problems in periodic setting to \cite{bonder2002nonlinear}.

The rest of this paper is organized as follows. In Section \ref{Sec setup}, we formulate the problem and state the main result, Theorem \ref{th:main}. 
In Section \ref{sec:cell-prob}, we introduce an auxiliary spectral problem on a periodicty cell and prove the existence of a principle positive eigenvalue in Lemma \eqref{positive-eignfn_condition}.
The regularity of the principal eigenvalue and the eigenfunction is discussed in Lemma \ref{lemm:Reqularity_mu}. Section \ref{sec: proof main theorem} is devoted to the proof of the main theorem. After factorizing the original eigenfunctions with the positive cell eigenfunction of \eqref{eq:cell-problem} in Section \ref{sec factorization and estimates eigenvalue}, we rescale the factorized problem and derive a priori estimates in Section \ref{sec rescaling and apriori estimate}. In Section \ref{sec passage of limit}, we pass to the limit using the two-scale convergence in spaces with singular measures. The definition of the two-scale convergence in spaces with singular measures is given in \ref{Appendix A}, and a mean-value property for the oscillating functions is given in \ref{Appendix B}.

\section{Problem setup}
\label{Sec setup}
Let $Q$ be a bounded $C^{2,\alpha}$ domain in $\mathbb{R}^{d-1}$, $d\ge 2$, with a boundary $\partial Q$. The points in $\mathbb{R}^d$ are denoted $x = (x_1,x')$, where $x' = x_2,...,x_d$. For a small parameter $\ve>0$, denote by $\Omega_\ve=[-1,1] \times (\ve Q)$ a thin rod with the lateral boundary $\Sigma_\ve = (-1,1) \times \partial (\ve Q)$ and the bases $S_{\ve, \pm} = \{\pm 1\} \times (\ve Q)$. In the cylinder $\Omega_\ve$, we consider the following spectral problem:
\begin{equation}
\label{orig-prob}
\left\{
\begin{array}{lcr}
\displaystyle
- \mathrm{div}\Big(a^\ve(x) \nabla u^\ve(x)\Big)=
\lambda^\ve \, \rho^\ve (x)\,  u^\ve(x), \quad \hfill x \in \Omega_\ve,
\\[2mm]
\displaystyle
a^\ve(x) \, \nabla u^\ve(x) \cdot n = 0, \quad \hfill x \in \Sigma_\ve,
\\[2mm]
\displaystyle
u^\ve(-1,x') = u^\ve(1,x')=0, \quad \hfill x'\in \ve Q
\end{array}
\right.
\end{equation}
with the $d\times d$ matrix $a^\ve$ and the scalar weight $\rho^\ve$:
\begin{align}
\begin{split}
&a^\ve(x) = a\big(x_1, \frac{x}{\ve}\big), \quad
\rho^\ve(x) = \rho \big(x_1, \frac{x}{\ve}\big).
\end{split}
\end{align}
We assume the following  conditions hold:
\begin{itemize}
\item[$\bf(H1)$]
$a_{ij}(x_1,y), \rho(x_1,y)\in C^{1,\alpha}([-1,1]; C^{0,\alpha}(\overline{Y}))$ for some $\alpha>0$; where $Y = (0,1] \times Q$ is the periodicity cell with the lateral boundary $\Sigma=(0,1)\times \partial Q$.
\item[$\bf(H2)$]
The functions $a_{ij}(x_1,y)$ and $\rho(x_1,y)$ are $1$-periodic in $y_1$. 
\item[$\bf(H3)$]
The matrix $a(x_1,y)$ is symmetric and satisfies the uniform ellipticity condition, that is for any $x_1 \in [-1,1]$ and $y \in Y$, and for some $\Lambda>0$ it holds
$$
\sum \limits_{i,j=1}^{d} a_{ij}(x_1, y) \xi_i \xi_j \geq \Lambda |\xi|^2, \quad \xi \in \mathbb{R}^d.
$$
\item[$\bf(H4)$]
The weight function $\rho(x_1,y)$ changes sign, that is for any $x_1 \in [-1,1]$ the sets $\{y \in Y \,:\,\rho(x_1,y)<0\}$ and $\{y \in Y \,:\,\rho(x_1, y)>0\}$ have positive Lebesgue measures. 
\item[$\bf(H5)$] $\int_Y \rho(x_1,y)\, dy <0$ for all $x_1\in[-1,1]$.
\end{itemize}
The weak formulation of problem \eqref{orig-prob} reads: Find {$\lambda^\ve \in \mathbb{C}$} (eigenvalues) and $u^\ve \in H^1(\Omega_{\ve}) \setminus \{0\}$ (eigenfunctions) such that $u^\ve(\pm 1, x')= 0$ and
\begin{equation}
\label{var-orig-prob}
(a^\ve \, \nabla u^\ve, \nabla v)_{L^2(\Omega_{\ve})} = \lambda^\ve \, (\rho^\ve \, u^\ve, v)_{L^2(\Omega_{\ve})},
\end{equation}
where $(\cdot, \cdot)_{L^2(\Omega_{\ve})}$ denotes the standard scalar product in $L^2(\Omega_{\ve})$.

The next lemma characterizes the spectrum of problem \eqref{orig-prob}.
\begin{lemma}
\label{existence_eigenvalue_orig prob}
Under the assumptions $\bf(H1)-(H4)$, the spectral problem \eqref{orig-prob} has a real and discrete spectrum that consists of two infinite sequences
$$
\begin{array}{c}
0 < \lambda_1^{\ve, +} \leq \lambda_2^{\ve, +} \leq \ldots \leq \lambda_j^{\ve, +} \leq \ldots \to + \infty,
\\[2mm]
0 > \lambda_1^{\ve, -} \geq \lambda_2^{\ve, -} \geq \ldots \geq \lambda_j^{\ve, -} \geq \ldots \to - \infty.  
\end{array}
$$
The corresponding eigenfunctions $u_j^{\ve,\pm}$ may be chosen to satisfy the orthogonality and normalization condition{
\begin{equation}
\label{norm cond u^ve}
(u_i^{\ve, \pm}, u_j^{\ve, \pm})_{L^2(\Omega_{\ve})} = \ve^{1/2}\ve^{d-1} \, |Q| \, \delta_{ij},
\end{equation}}
where $|Q|$ is the Lebesgue measure of $Q$ and $\delta_{ij}$ is the Kronecker delta.
\end{lemma}
\begin{remark}
The reason for choosing this particular normalization \eqref{norm cond u^ve} is to get the eigenfunctions of the rescaled problem \eqref{eq:rescle-final} and the limit problem \eqref{eq:eff-prob} normalized in a standard way (without the small parameter $\ve$).
\end{remark}
A proof of Lemma \ref{existence_eigenvalue_orig prob} can be found in \cite{pyatkov1998indefinite} or Lemma 1 in \cite{nazarov2011homogenization}.
In the present work, we study the asymptotics of the positive eigenvalues $\lambda_j^{\ve, +}$, as $\ve \to 0$, under the assumption (H5) that the local average of the weight function is negative $\int_Y \rho(x_1, y)\, dy <0$. 
In this case, as will be shown below, the positive eigenvalues grow as $\ve \to 0$, and the asymptotics can be obtained by using a special factorization with a positive eigenfunction of the auxiliary spectral problem stated on the periodicity cell $Y$.
Namely, we will use a positive principal eigenvalue $\mu(x_1)$ of the auxiliary spectral cell problem with sign-changing weight, for each $x_1\in[-1,1]$:
\begin{align}
\label{eq:cell-problem-0}
\begin{cases}
&-{\rm div}_y (a(x_1, y) \nabla_y \psi(x_1, y)) = \mu(x_1) \rho(x_1, y)\psi(x_1, y), \quad \hfill y\in Y, \\
&a(x_1, y) \nabla_y \psi(x_1, y) \cdot \nu = 0, \quad \hfill y\in \Sigma,\\
&y_1 \mapsto \psi(x_1, y_1, y') \quad \mbox{is} \,\, 1-\mbox{periodic}.
\end{cases}
\end{align}
We say that $\mu$ is a principal eigenvalue of \eqref{eq:cell-problem-0} if it possesses a unique strictly positive eigenfunction. Obviously, $\mu=0$ is one such principal eigenvalue with a constant eigenfunction. For our purpose, we are interested in a strictly positive principal eigenvalue $\mu$ and a non-constant eigenfunction $\Psi$. We will prove that in the case $\int_Y \rho(x_1, y)\, dy <0$ such an eigenvalue exists (see Lemma \ref{lm:existence-mu}).  

In what follows we assume that the principal positive eigenvalue $\mu(x_1)$ of \eqref{eq:cell-problem-0} satisfies the following assumption.\\[2mm]
\textbf{(H6)} The principal positive eigenvalue $\mu(x_1)$ of problem \eqref{eq:cell-problem} has a unique minimum point at $x_1 = 0$ and $\mu''(0)>0$.\\[1mm]

The main result of the paper is contained in the following theorem.
\begin{theorem}
    \label{th:main}
    Let the hypotheses (H1)--(H6) hold and denote $(\mu, \Psi)$ the principal eigenpair of \eqref{eq:cell-problem-0}. Then, for any $j$, we have the following convergence result:
\begin{align*}
    &\lambda_{j}^{\ve,+}= \frac{\mu(0)}{\ve^2} + \frac{\nu_j}{\ve} + o(\ve^{-1}), \quad \ve \to 0,\\
    &\ve^{-\frac{d-1}{2}} \| u_{j}^{\ve, +} - \Psi(0, \frac{x}{\ve}) \, v_j(\frac{x_1}{\sqrt \ve}) \|_{L^{2}(\Omega_{\ve})} \rightarrow 0, \quad \ve \to 0,
\end{align*}
where $(\nu_j, v_j)$ is the $j$th eigenpair of the harmonic oscillator
\begin{align}
\label{eq:eff-prob}
        &- \displaystyle(a^\eff v')' 
        + \big(c^\eff + \frac{1}{2}\mu''(0) x_1^2\big) v= \nu \, v, \quad x_1 \in \mathbb R.
\end{align}
The effective coefficients in \eqref{eq:eff-prob} are defined by
\begin{align}
\label{def:eff-coeff}
&a^\eff =\frac{1}{|Y|} \int_{Y} a(0, \zeta) \nabla(\zeta_{1} + N_{1}(\zeta)) \cdot \nabla(\zeta_{1} + N_{1}(\zeta)) \, d\zeta,\\
&c^\eff = \frac{1}{|Y|}\int_{Y}\Big( (a \nabla_{x} \Psi)(0,\zeta) \cdot \nabla_{\zeta}\Psi(0, \zeta) - \mathrm{div}_{x}(a \nabla_{\zeta} \Psi)(0, \zeta) \Psi(0, \zeta) \Big)\, d\zeta,
\end{align}
where $N_{1}$ satisfies the cell problem
\begin{align*}
\begin{cases}
&\mathrm{div}((a\, \Psi^2)(0, \zeta) \nabla (N_{1}(\zeta) + \zeta_{1})) = 0, \quad \hfill \zeta\in Y,\nonumber\\
&(a\, \Psi^2)(0, \zeta) \nabla (N_{1} + \zeta_{1}) = 0, \quad \hfill \zeta\in \Sigma,\\
&N_1(\cdot, \zeta') \,\, \mbox{is} \,\, 1-\mbox{periodic}.
\end{cases}
\end{align*}
\end{theorem}
In \cite{pankratova2011homogenization}, problem \eqref{orig-prob} has been studied under different assumptions on the local weight average. For the reader's convenience, we summarize the results about the spectral asymptotics in the cases studied in \cite{pankratova2011homogenization}. We denote 
\begin{align*}
&\mathcal{A}_y \, u = - {\rm div}_y (a(x_1, y) \nabla_y u),\quad
\mathcal{B}_y \, u = a(x_1, y) \nabla_y u\cdot n,\\
&\langle \rho(x_1,.) \rangle = \int_Y \rho(x_1, y)\, dy.
\end{align*}

\noindent
\textbf{Case 1:} If $\langle \rho(x_1,.) \rangle > 0$ for all $x_{1} \in [-1,1]$. Then, for any $j$,
\begin{align*}
    &\lambda_{j}^{\ve,+} \rightarrow \lambda_{j}, \quad \ve^{-\frac{d-1}{2}} \| u_{j}^{\ve,+} - u_{j} \|_{L^{2}(\Omega_{\ve})} \rightarrow 0, \quad \ve \rightarrow 0,
\end{align*}
where $(\lambda_j, u_j)$ are solutions of the limit problem
\begin{align*}
    \begin{cases}
        &- \Big( a^{\eff}(x_1) u'(x_1)\Big)' = \lambda^{0} \displaystyle\langle \rho(x_1,.) \rangle u^{0}(x_1), \quad x_1 \in (-1,1),\\
        & u^{0}(\pm 1) = 0.
    \end{cases}
\end{align*}
Here 
\begin{align}
\label{eq:a-eff}
    a^{\eff}(x_1) = \displaystyle\int_{Y} a_{1j}(x_1,y) (\delta_{1j} + \partial_{y_j} N^{1,1}(x_1,y)) \, dy,
\end{align}
and $N^{1,1}$ solves, for each $x_1\in[-1,1]$, the cell problem
\begin{align}
\label{eq:N^1,1-rho>0}
  \begin{cases}
    &\mathcal{A}_{y} N^{1,1}(x_1,y) =\displaystyle\mathrm{div}_{y} a_{.1}(x_1,y), \quad y\in Y,\\
    &\mathcal{B}_{y} N^{1,1}(x_1,y) = -a_{.1}(x_1,y)\cdot n, \quad y\in \Sigma,\\
    &y_1 \mapsto N^{1,1}(x_1,y_1, y') \,\, \mbox{is} \,\, 1- ~\text{periodic.}
    \end{cases}
\end{align}
Thus, the homogenization of the positive part of the spectrum in the case of positive average of $\rho$ is classical.\\[2mm]
\noindent
\textbf{Case 2:} If $\langle \rho(x_1,.) \rangle = 0$ for all $x_{1} \in [-1,1]$, then, for any $j$,
\begin{align*}
    &\ve \lambda_{j}^{\ve,\pm} \to \nu_{j}^{\pm}, \quad \ve^{-\frac{d-1}{2}} \| u_{j}^{\ve, \pm} - v_{j}^{\pm} \|_{L^{2}(\Omega_{\ve})} \rightarrow 0, \quad \ve \rightarrow 0,
\end{align*}
where $(\nu_j^{\pm}, v_j^{\pm})$ are the $j$th eigenpairs (with positive and negative eigenvalues) of the following  quadratic operator pencil:
\begin{equation}
\left\{
\begin{array}{l}
\displaystyle
-\Big(a^{\eff}(x_1) v'(x_1)\Big)' +
\nu \, \mathbf{B}(x_1)\, v(x_1)
- \nu^2\, \mathbf{C}(x_1)\, v(x_1)=0, \quad x_1 \in (-1,1),
\\[2mm]
\displaystyle
v(- 1)=v(1)=0.
\end{array}
\right.
\end{equation}
The functions $\mathbf{B}(x_1), \mathbf{C}(x_1)>0$ are defined by
\begin{align*}
&\mathbf{C}(x_1)= \int_Y (a\, \nabla_y N^{1,0},\nabla_y N^{1,0}) \, dy,
&\mathbf{B}(x_1)= \frac{\partial}{\partial x_1} \int_Y a\, \nabla_y N^{1,1}\cdot\nabla_y N^{1,0} \, dy.
\end{align*}
The function $N^{1,1}$ solves \eqref{eq:N^1,1-rho>0} and $N^{1,0}$ is a solution of 
\begin{equation}
\label{N^1,0-rho=0}
\left\{
\begin{array}{lcr}
\displaystyle
\mathcal{A}_y N^{1,0}(x,y) = \rho(x_1, y), \quad y \in Y,
\\[2mm]
\displaystyle
\mathcal{B}_y N^{1,0}(x_1,y)= 0, \quad y \in \Sigma,
\\[2mm]
N^{1,0}(x_1,y)\,\, \mbox{is}\,\, 1-\mbox{periodic in} \,\, y_1.
\end{array}
\right.
\end{equation}

\noindent
\textbf{Case 3:} If $\langle \rho(x_1,.) \rangle$ changes sign, then for any $j$, 
\begin{align*}
    &\lambda_{j}^{\ve,\pm} \rightarrow \lambda_{j}^{\pm}, \quad \ve^{-\frac{d-1}{2}} \| u_{j}^{\ve, \pm} - u_{j}^{\pm} \|_{L^{2}(\Omega_{\ve})} \rightarrow 0, \quad \ve \rightarrow 0,
\end{align*}
where $(\lambda_j^{\pm}, u_j^{\pm})$ are solutions of the limit problem
\begin{align}
\label{eq:eff-prob-as-expansions}
    \begin{cases}
        &- \Big( a^{\text{eff}}(x_1) u'(x_1)\Big)' = \lambda^{0} \displaystyle\langle \rho(x_1,.) \rangle u^{0}(x_1), \quad x_1 \in (-1,1),\\
        & u^{0}(\pm 1) = 0,
    \end{cases}
\end{align}
with the effective coefficients defined by \eqref{eq:a-eff}.

To summarize, one case which was not studied in \cite{pankratova2011homogenization} is {the asymptotics of the positive part of the spectrum for the negative average of the weight function}. The case of the negative part of the spectrum for a weight function with positive average is similar.

\begin{remark}
The standard asymptotic ansatz does not provide any information about the asymptotics of positive eigenvalues in the case $\langle \rho(x_1,\cdot)\rangle<0$. Indeed, let us look for a solution $(\lambda^\ve, u^\ve)$ of problem \eqref{orig-prob} in the form
\begin{equation}
\label{ansatz-1}
\begin{aligned}
    u^\ve (x) &= u(x_1) + \ve u^1(x_1,y) + \ve^2 u^2(x_1,y) +\cdots,~~~ y= \frac{x}{\ve},\\
    \lambda^\ve &= \lambda^0 + \ve\lambda^1 + \cdots,
\end{aligned}
\end{equation}
where the unknown functions $u^k(x_1,y)$ are 1-periodic in $y_1$. Substituting ansatz \eqref{ansatz-1} into \eqref{orig-prob}, applying the chain rule, and collecting power-like with respect to $\ve$ terms, we obtain a cascade of problems for $u^k$.

In particular, the right-hand side in the problem for $u^1$ suggests to choose $u^{1}(x_1,y) = N^{1,1}(x_1,y) u'(x_1)$,
where $N^{1,1}$ solves \eqref{eq:N^1,1-rho>0} for each $x_1 \in [-1,1]$. 
The compatibility condition for the problem for $u^2$ gives an equation for $u^0$:
\begin{align}
\label{eq:eff-prob-as-expansions}
    \begin{cases}
        &- \displaystyle\frac{d}{d x_1} \Big( a^{\text{eff}}(x_1) \frac{d u^{0}(x_1)}{d x_1} \Big) = \lambda^{0} \displaystyle\langle \rho(x_1,.) \rangle u^{0}(x_1), \quad x_1 \in (-1,1),\\
        & u^{0}(\pm 1) = 0.
    \end{cases}
\end{align}
Here $a^\eff$ is defined by \eqref{eq:a-eff}.
Since $\langle \rho(x_1,\cdot) \rangle < 0$, \eqref{eq:eff-prob-as-expansions} possesses only negative eigenvalues and thus provides no information about the positive eigenvalues $\lambda_j^\ve$ of the original problem \eqref{orig-prob}. In the following sections, we will show that the positive eigenvalues of \eqref{orig-prob} tend to infinity, as $\ve \to 0$, and derive the correct limit spectral problem.
\end{remark}

\section{Auxiliary spectral cell problem}
\label{sec:cell-prob}
In order to estimate the positive eigenvalues $\lambda_j^{\ve, +}$ of \eqref{orig-prob}, we will use a positive principal eigenvalue $\mu(x_1)$ of the auxiliary spectral cell problem with sign-changing weight, $x_1\in[-1,1]$:
\begin{align}
\label{eq:cell-problem}
\begin{cases}
&-{\rm div}_y (a(x_1, y) \nabla_y \psi(x_1, y)) = \mu(x_1) \rho(x_1, y)\psi(x_1, y), \quad \hfill y\in Y, \\
&a(x_1, y) \nabla_y \psi(x_1, y) \cdot \nu = 0, \quad \hfill y\in \Sigma,\\
&y_1 \mapsto \psi(x_1, y_1, y') \quad \mbox{is} \,\, 1-\mbox{periodic}.
\end{cases}
\end{align}
We say that $\mu$ is a principal eigenvalue of \eqref{eq:cell-problem} if it possesses a strictly positive eigenfunction $\psi(x_1, \cdot) \in H^1(Y)$. Obviously, $\mu=0$ is one such principal eigenvalue with eigenfunction $\Psi=1$. We will prove in Lemma \ref{lm:existence-mu} that there exists a positive principal eigenvalue.  

The Neumann problem for the Laplace operator has been studied, for example, in \cite{brown1980existence}, \cite{lou2006minimization}. In particular, in \cite{brown1980existence} it has been shown that the Neumann problem in a bounded domain $\mathcal O$
\begin{align*}
\begin{cases}
-\Delta \psi = \mu \rho \psi,\quad y\in \mathcal O,\\
\nabla \psi \cdot n =0, \quad y\in \partial \mathcal O,
\end{cases}
\end{align*}
has a positive principal eigenvalue if and only if the measure of the set $\{y:\, \rho(y)>0\}$ is positive and the average of the weight is negative $\int_{\mathcal O} \rho \, dy <0$. 
In this section, we will prove the corresponding result for \eqref{eq:cell-problem}.

Consider the operator 
$\mathcal A_y v = - \mathrm{div}_{y} \Big( a(x_{1}, y) \nabla_{y} v\Big)$
with the domain 
$$D(\mathcal A_y) = \{ v \in H^{2}(Y): a\nabla_y v\cdot \nu\big|_{\Sigma} =0, \, y_{1} \mapsto v(x_{1},y_{1},y') \text{ is } 1\text{-periodic}\}.$$ 
We define
   $$ Q_{\mu}(v) = (\mathcal A_y v, v)_{L^{2}(Y)} - \mu (\rho v, v)_{L^{2}(Y)}. $$
\begin{lemma}
\label{positive-eignfn_condition}
    If there is a positive eigenfunction corresponding to an eigenvalue $\mu(x_{1})$ of \eqref{eq:cell-problem}, then $Q_{\mu}(v) \ge 0$ for all $v \in D(\mathcal A_y)$.
\end{lemma}
\begin{proof}
If $u>0$ is an eigenfunction corresponding to $\mu$ of \eqref{eq:cell-problem}, then $u$ is also an eigenfunction corresponding to the eigenvalue $r = k$ of the operator
\begin{align}
\label{spctrl_cell_prob_shifted}
       T_k \psi = T_0 \psi + k\psi :=  \mathcal A_y \psi - \mu \rho \psi + k \psi &= r \psi \quad \mbox{in } Y,\\ 
        \mathcal B_y \psi = a \nabla_{y} \psi\cdot \nu
 &= 0 \quad \mbox{on } \Sigma,\nonumber\\
 y_1 \mapsto \psi(y_1, y') \, &\mbox{ is 1-periodic}. \nonumber
\end{align}
The spectrum of \eqref{spctrl_cell_prob_shifted} is discrete for sufficiently large $k$ and consists of a countable number of eigenvalues. The first eigenvalue $r_1$ is simple, and the corresponding eigenfunction $\psi_1$ can be chosen positive. 
Clearly, $(r_{i}, \psi_{i})$ is an eigenpair of $T_{k}$ \eqref{spctrl_cell_prob_shifted} if and only if $(r_i-k, \psi_{i})$ is an  eigenpair of $T_0$.
Because the eigenfunction $u$ is positive,
\begin{align*}
    (u , \psi_{1} )_{L^{2}(Y)} = \int_{Y} u \psi_{1} \, dx > 0.
\end{align*} 
As $u$ is an eigenfunction for some eigenvalue in the spectrum of $T_{k}$, $u$ must be an eigenfunction corresponding to $r_{1}$ due to the fact that the eigenfunctions corresponding to distinct eigenvalues are orthogonal. Then $u \in \text{Span}(\psi_{1})$, as $r_{1}$ is simple. As $T_{k} u = k u$ we have $r_{1} = k$. Furthermore, $(r_{1} - k)=0$ is a simple eigenvalue for $T_{0}$, and the corresponding eigenfunction $\psi_{1}$ does not change sign.\\
Since $T_{k}v = T_{0}v + k\,v$, then for any test function $v$, {by the spectral theorem},
\begin{align*}
   (T_{0} v,v)_{L^{2}(Y)} + k \| v \|^{2}_{L^{2}(Y)} = (T_{k} v, v)_{L^{2}(Y)} &\geq k \|v\|^{2}_{L^{2}(Y)}.
\end{align*}
Since $Q_{\mu} (v) = (T_{0} v,v)_{L^{2}(Y)}$, we have
$Q_{\mu} (v) \geq 0$ for all $v \in D(\mathcal A_y)$.
\end{proof}
The proof of the existence of a positive principal eigenvalue with a non-constant eigenfunction relies on the following statement (see Lemma 3.9 in \cite{brown1980existence}).
\begin{lemma}
\label{lm:estimate}
Let the hypotheses (H1)-(H5) hold.
If $\langle\rho(x_1,\cdot\rangle)=\int_Y \rho \, dy < 0$, there exist $\delta > 0, \gamma > 0$ such that for all $ \psi \in D(\mathcal A_y)$ satisfying 
\begin{align}
\label{eq:condition}
\int_{Y} \rho \psi^{2} \, dy > - \gamma \int_{Y} \psi^{2} \, dy,
\end{align}
we have
\begin{align}
\label{eq:inequality}
    \int_{Y} \psi^{2} \, dy \le \frac{1}{\delta}\int_{Y} |\nabla_{y} \psi|^2\, dy.
\end{align}
\end{lemma}
It is clear that for a constant non-zero $\psi$, the Friedrichs inequality \eqref{eq:inequality} is not satisfied. However, condition \eqref{eq:condition} does not allow constant $\psi$ because of the negative local average of $\rho$. 
The existence of a nontrivial principal eigenvalue is ensured by the following lemma. 
\begin{lemma}
\label{lm:existence-mu}
Let hypotheses (H1)--(H4) hold.
    \begin{itemize}
    \item[(a)] If $\int_Y \rho \, dy \ge 0$, then $\mu=0$ is the only nonnegative eigenvalue for which the corresponding eigenfunction does not vanish. 
    \item[(b)] If $\int_Y \rho \, dy<0$, there exists a unique positive principal eigenvalue for which the corresponding eigenfunction does not vanish in $Y$, and
    \begin{align}
    \label{eq:mu1}
    \mu_1(x_1) = \inf_{\substack{v \in D(\mathcal A_y) \\ (\rho v , v)_{L^{2}(Y)}> 0}} \frac{\int_Y a(x_1,y) \nabla_{y} \psi \cdot \nabla_{y} \psi \, dy}{\int_Y \rho(x_1,y) \psi^2\, dy},
    \end{align}
    where 
    $$D(\mathcal A_y) = \{ v \in H^{2}(Y): a\nabla_y v\cdot \nu\big|_{\Sigma} =0, \, y_{1} \mapsto v(x_{1},y_{1},y') \text{ is } 1\text{-periodic}\}.
    $$ 
The eigenfunction $\Psi$ corresponding to $\mu_1$ can be chosen positive and normalized by $(\rho \Psi, \Psi)_{L^2(Y)}=1$.
\end{itemize}
\end{lemma}
\begin{proof}
In the current proof, $x_1$ is a fixed parameter, and we omit it for brevity.\\
(a) 
Let $\mu_1$ be defined by \eqref{eq:mu1}. Clearly, $\mu_{1} \ge 0$.
Let us prove that $\mu_{1} = 0$ by showing that for any $\mu>0$ there exists $v \in D(\mathcal A_y)$ such that $Q_{\mu}(v)<0$, when $\int_{Y}\rho\,  dy \geq 0$. This will contradict Lemma \eqref{positive-eignfn_condition}.

Suppose first that $\int_{Y} \rho \, dy > 0$.
Then we can take $v=1$ to get
\begin{align*}
Q_{\mu}(v) &= -\mu \int_{Y} \rho \, dy < 0.
\end{align*} 
If $\int_{Y} \rho \, dy = 0$,
we choose any $w \in D(\mathcal A_y)$ such that $\int_{Y} \rho w \, dy > 0$. For a sufficiently small $s>0$ we have 
\begin{align*}
    Q_{\mu}(1 + sw) 
    = s^{2} Q_{\mu}(w) - 2 s \mu \int_{Y} \rho w \, dy < 0.
\end{align*}
So $\mu_{1}$ cannot be positive, and therefore, $\mu_{1} = 0$ is the only principal eigenvalue in the case when $\int_Y \rho(x_1, y)dy \ge 0$.\\

\noindent
(b) $\mu_{1}$ given by \eqref{eq:mu1} is positive.
Indeed, by Lemma \ref{lm:estimate} and  (H3), we have
\begin{align*}
    \mu_{1} = \inf_{\substack{\psi \in D(\mathcal A_y) \\ (\rho \psi, \psi)_{L^{2}(Y)}> 0}}  \frac{\int_{Y} a(x_{1},y) \nabla_{y} \psi \cdot \nabla_{y} \psi \, dy}{\int_{Y} \rho \psi^{2} \, dy} \geq \frac{\delta \Lambda}{\|\rho\|_{L^{\infty}(Y)}} > 0.
\end{align*}
Next, we prove that $\mu_{1}$ is a positive eigenvalue of \eqref{eq:cell-problem}.
Consider the shifted eigenvalue problem
\begin{align}
\label{new_spctrl_prob}
\begin{cases}
    T_0 \psi = \mathcal A_y \psi - \mu_{1} \rho \psi = \lambda \psi \quad \mbox{in}~ Y,\\
    \mathcal B_y \psi = 0 \quad \mbox{on}~ \Sigma, \nonumber\\
    y_1\mapsto \psi \,\,\mbox{1-periodic}.
\end{cases}
 \end{align}
It is clear from \eqref{new_spctrl_prob} that $\mu_{1}$ is an eigenvalue of \eqref{eq:cell-problem} if and only if $\lambda=0$ is an eigenvalue of $T_{0}$.
By the definition of $\mu_1$, $(T_{0} v, v)_{L^{2}(Y)} \geq 0$ for all $v \in D(\mathcal A_y)$.
The least eigenvalue of $T_{0}$ is given by
\begin{align*}
    \alpha_{1} &= \inf \{ (\mathcal A_y v, v)_{L^2(Y)} - \mu_{1}(\rho v, v)_{L^2(Y)}: v \in D(\mathcal A_y), (\rho v,v)_{L^2(Y)}=1 \}\\
    &= \inf \{ Q_{\mu_{1}}(v) : v \in D(\mathcal A_y), (\rho v,v)_{L^2(Y)}=1\}.
\end{align*}
Then $\alpha_{1} \geq 0$ by the defintion of $\mu_1$. Moreover, 
there exists a sequence $\{v_{n}\} \subset D(\mathcal A_y)$ such that $\int_{Y} \rho v_{n}^{2} \, dy = 1$ and 
\begin{align*}
\displaystyle\lim_{n \rightarrow \infty} \frac{(\mathcal A_y v_{n}, v_{n})_{L^2(Y)}}{(\rho v_{n}, v_{n})_{L^{2}(Y)}} = \lim_{n \rightarrow \infty} (\mathcal A_y v_{n}, v_{n})_{L^{2}(Y)} = \mu_{1}.
\end{align*}
Then $ \displaystyle\lim_{n \rightarrow \infty} Q_{\mu_{1}}(v_{n}) = \mu_{1} - \mu_{1} = 0$, and  
\begin{align*}
    \alpha_{1} = \displaystyle
    \inf_{\substack{v \in D(\mathcal A_y) \\ (\rho v, v)_{L^{2}(Y)}> 0}} Q_{\mu_{1}}(v) \leq \liminf_{n \rightarrow \infty} Q_{\mu_{1}}(v_{n}) = 0.
\end{align*}
This yields that the smallest eigenvalue of $T_0$ is $\alpha_{1} = 0$. By the Krein-Rutman theorem, it is simple and the corresponding eigenfunction can be chosen positive in $Y$.
This shows that $\mu_{1}>0$ is an eigenvalue of \eqref{eq:cell-problem} that is simple, and the corresponding eigenfunction can be chosen positive. It is left to show that there are no other principal eigenvalues.

Take $\mu > 0$ such that $\mu \neq \mu_{1}$. Let us show that $\mu$ is not an eigenvalue with a non-negative eigenfunction.

Suppose first that $0 < \mu < \mu_{1}$. 
We will show that there exists $\alpha > 0$, depending on $\mu$, such that $Q_{\mu}(v) \geq \alpha \|v\|^{2}_{L^{2}(Y)}$ for all $v \in D(\mathcal A_y)$.
For $\mu = (1 - s) \mu_{1}$, $0 <s< 1$, we have
\begin{align*}
    Q_{\mu}(v) &= (\mathcal A_y v, v)_{L^{2}(Y)} - \mu (\rho v, v)_{L^{2}(Y)}\\
    &= \frac{\mu}{\mu_{1}} Q_{\mu_{1}}(v) + (1 - \frac{\mu}{\mu_{1}}) (\mathcal A_y v, v)_{L^{2}(Y)} \geq s ~ (\mathcal A_y v, v)_{L^{2}(Y)}. 
\end{align*}
Let $\delta, \gamma > 0$ be such as in Lemma \ref{lm:estimate}. Then for any $v \in D(\mathcal A_y)$ such that $ \int_{Y} \rho v^{2} \, dy > - \gamma \int_{Y} v^{2} \, dy$, we have
\begin{align*}
    Q_{\mu}(v) 
     & \geq s \delta \Lambda \|v\|^{2}_{L^{2}(Y)}.
\end{align*}
If $ \int_{Y} \rho v^{2} \, dy \leq - \gamma \int_{Y} v^{2} \, dy$, then $Q_\mu(v)$ is estimated in the following way:
\begin{align*}
    Q_{\mu}(v) 
    &\geq \Lambda \|\nabla_{y} v\|^{2}_{L^{2}(Y)} - \mu (\rho v, v)_{L^{2}(Y)}\\
    &\geq - \mu (\rho v, v)_{L^{2}(Y)} \geq \mu \gamma \|v\|^{2}_{L^{2}(Y).}
\end{align*}
Therefore, there exists no non-trivial $v \in D(\mathcal A_y)$ for $\mu < \mu_{1}$ such that $Q_{\mu}(v) = 0$. So $\mu$ is not an eigenvalue of \eqref{eq:cell-problem}.

Next, suppose that $\mu$ is such that $0< \mu_{1} < \mu$. Then $Q_{\mu}(v) \geq 0$ for all $v \in D(\mathcal A_y)$. Indeed, there exists $v$ such that
\begin{align*}
    &\frac{(\mathcal A_y v,v)_{L^2(Y)}}{(\rho v, v)_{L^2(Y)}} < \mu \Rightarrow
    Q_{\mu}(v) = \int_{Y} a \nabla_{y}v \cdot \nabla_{y}v \, dy - \mu \int_{Y} \rho v^{2} \, dy < 0.
\end{align*}
So $\mu$ cannot be an eigenvalue with a positive eigenfunction due to Lemma \ref{positive-eignfn_condition}.

\end{proof}

{The next lemma characterizes the regularity of $(\mu(x_1), \Psi(x_1, y))$. One of the key arguments we use is the positivity of the eigenfunction, and, thus, the proof is valid only for the principal eigenvalue.}
\begin{lemma}\label{lemm:Reqularity_mu}
Let $\mu=\mu(x_1)$ be the principle non-zero eigenvalue of \eqref{eq:cell-problem} with the corresponding eigenfunction $\Psi=\Psi(x_1, y)$. Under assumptions (H1)--(H4), we have 
\begin{align}
\mu\in C^{1,\alpha}([-1,1]), \quad 
\Psi\in C^{1,\alpha}([-1, 1]; C^{1, \alpha}(\overline Y ) \cap H^1(Y)).
\end{align}
\end{lemma}

\begin{proof}
{The classical elliptic regularity of the solutions to elliptic equations ensures that $\Psi(x_1, \cdot)\in H^1(Y) \cap C^{1,\alpha}(\overline Y)$, for each $x_1 \in [-1,1]$ (see Corollary 1.4 in \cite{vita2022boundary}).}
{Under the coercivity and boundedness conditions on the coefficient $a(x_1,y)$, the bilinear forms $a_{x_1}(v,v)=(\mathcal A_y v, v)_{L^2(Y)}$ and $b_{x_1}(v,\varphi)=(\rho(x_1,\cdot)v, \varphi)_{L^2(Y)}$ are Fr\'{e}chet differentiable with respect to $x_1$} (see Theorem 2.4.1 in \cite{komkov1986design}), and the corresponding differentials are
\begin{align*}
a_{\delta x_1}'(v, \varphi)&= \delta x_1 \int_Y \partial_{x_1}a(x_1, y)\nabla v\cdot \nabla \varphi\, dy,\\
b_{\delta x_1}'(v,\varphi)&= \delta x_1 \int_Y \partial_{x_1}\rho(x_1, y)\, v\, \varphi\, dy.
\end{align*}

From the coercivity and differentiability of the bilinear forms it follows that the operator $\mathcal A_y$ (given by its bilinear form) {with the domain $D(\mathcal A_y)$, which is dense in $L^2(Y)$, possesses an inverse $\mathcal A_y^{-1}$ }that is Fr\'echet differentiable (Theorem 2.4.2 in \cite{komkov1986design}). In addition, the operator $B_{x_1}$ defined by the bilinear form $(B_{x_1} v, \varphi)_{L^2(Y)}=(\rho(x_1,\cdot)v, \varphi)_{L^2(Y)}$ is bounded from $L^2(Y)$ into itself, and, thus, is Fr\'echet differentiable with respect to $x_1$. Thus, $\mathcal A_y^{-1} B_{x_1}$ is also Fr\'echet differentiable. The eigenvalue problem \eqref{eq:cell-problem} can be equivalently rewritten as
\begin{align}
\label{eq:equiv-eigenvalue}
\mathcal A_y^{-1} B_{x_1} \Psi = \frac{1}{\mu(x_1)}\Psi.
\end{align}
Then $1/\mu(x_1)$ is continuous with respect to $x_1$, and, furthermore, by Theorem 2.5.2 in \cite{komkov1986design}, since the principal eigenvalue is simple, it is differentiable, and the derivative is given by
\begin{align}
\label{eq:m'}
\mu'(x_1)&= \int_Y \partial_{x_1} a(x_1,y)\nabla_y \Psi(x_1,y)\cdot \nabla_y \Psi(x_1,y)\, dy\\
&-\mu(x_1)\int_Y \partial_{x_1}\rho(x_1,y) \Psi(x_1,y)^2\, dy,\nonumber
\end{align}
where $\Psi$ is the corresponding eigenfunction normalized by $\int_Y \rho \Psi^2\,dy=1$.

Having proved that the eigenvalue $\mu(x_1)$ is differentiable with respect to $x_1$, we conclude that the corresponding eigenfunction is differentiable in $x_1$, as a solution to a boundary value problem with a given $\mu(x_1)$ (see Theorem 2.4.3 in \cite{komkov1986design}). 

Next, we prove the Lipschitz regularity of $\mu(x_1)$. Let us estimate $|\mu(\xi)-\mu(\eta)|$. 
For brevity, for two values $\xi, \eta \in [-1,1]$, we write
\begin{align}
\label{eq:L_xi}
L_{\xi}\Psi_{\xi} &:= -{\rm div}_y(a(\xi, y)\nabla_y \Psi(\xi, y))= \mu(\xi)\rho(\xi,y)\Psi(\xi, y)=:\mu_{\xi} \rho_{\xi} \Psi_{\xi},\\
\nonumber
L_{\eta}\Psi_{\eta} &:= -{\rm div}_y(a(\eta, y)\nabla_y \Psi(\eta, y))
= \mu(\eta)\rho(\eta,y)\Psi(\eta,y)
=:\mu_{\eta} \rho_{\eta} \Psi_{\eta}.
\end{align}
Then
\begin{align}
\label{eq:L_xi_Psi_eta}
L_\xi \Psi_\eta = (L_\xi-L_\eta) \Psi_\eta  + \mu_\eta \rho_\eta \Psi_\eta.
\end{align}
Taking the scalar product of \eqref{eq:L_xi_Psi_eta} with $\Psi_\xi$ in $L^2(Y)$, integrating by parts and using \eqref{eq:L_xi}, we obtain
\begin{align}
\label{eq:aux-1}
\mu_\xi(\rho_\eta \Psi_\eta, \Psi_\xi)_{L^2(Y)}
- \mu_\eta (\rho_\xi \Psi_\eta, \Psi_\xi)_{L^2(Y)}
=
((L_\xi-L_\eta) \Psi_\eta, \Psi_\xi )_{L^2(Y)}.
\end{align}
Using the linear approximations for $a_\eta, \rho_\eta$ and integrating by parts in \eqref{eq:aux-1}, we get
\begin{align}
\label{m_Holder-continuity}
(\mu_\xi-\mu_\eta) (\rho_\xi \Psi_\eta, \Psi_\xi)_{L^2(Y)}
&= 
(\partial_{x_1}a(\zeta, \cdot)(\eta - \xi) \nabla \Psi_\eta, \nabla \Psi_\xi)_{L^2(Y)} \\ 
& + (\partial_{x_1}a(\zeta, \cdot)(\eta - \xi) \nabla \Psi_\eta\cdot n, \Psi_\xi)_{L^2(\Sigma)} \nonumber\\
&- \mu_\eta(\partial_{x_1}\rho(\zeta, \cdot)(\eta - \xi) \Psi_\eta, \Psi_\xi)_{L^2(Y)}.\nonumber
\end{align}
Since $a_{ij}, \rho\in C^{1, \alpha}([-1,1]; C^{0,\alpha}(\overline Y))$, the partial derivatives $\partial_{x_1} a_{ij}, \partial_{x_1}\rho(x_1, y)$ are H\"older continuous on $[-1,1]$ and, thus, uniformly bounded. Moreover, {by the differentiability of $\mu(x_1)$ and $\Psi(x_1, y)$ in $x_1$, the normalization condition $\int_Y \rho \Psi^2\, dy=1$ implies then that for $|\xi-\eta|$ small enough, $\int_Y \rho_\xi \Psi_\xi \Psi_\eta\, dy \ge 1/2$.}
Thus, $\mu$ is Lipschitz continuous:
\begin{align}
\label{eq:est-mu-1}
|\mu(\xi)-\mu(\eta)|\le C |\xi-\eta|.
\end{align}
Differentiating \eqref{eq:cell-problem} with respect to $x_1$, we obtain:
\begin{align}
-&{\rm div}_y(a(x_1,y)\nabla_y (\partial_{x_1}\Psi(x_1, y))) - \mu(x_1) \rho(x_1,y) (\partial_{x_1}\Psi(x_1, y)) \nonumber\\
&\quad= {\rm div}_y (\partial_{x_1}a(x_1,y) \nabla_y \Psi(x_1,y))
+ \mu'(x_1) \rho(x_1, y) \Psi(x_1,y) \nonumber\\
&\qquad+\mu(x_1)\partial_{x_1} \rho(x_1,y) \Psi(x_1,y), \quad &y\in Y, \label{eq:aux-2}\\
&a(x_1, y)\nabla_y (\partial_{x_1}\Psi(x_1, y))\cdot n 
= -\partial_{x_1}a(x_1,y)\nabla \Psi(x_1,y)\cdot n, \quad &y\in \Sigma,\nonumber\\
&y_1 \mapsto \Psi(x_1, y_1, y') \,\mbox{is 1-periodic}.\nonumber
\end{align}
The eigenfunction $\Psi$ is normalized by $\int_Y \rho \Psi^2\, dy =1$ which implies 
\begin{align*}
&\partial_{x_1}\int_Y \rho(x_1,y) \Psi(x_1,y)^2\, dy\\
&= \int_Y \partial_{x_1}\rho(x_1,y) \Psi(x_1,y)^2\, dy
+ 2 \int_Y \rho(x_1,y) \partial_{x_1}\Psi(x_1,y) \Psi(x_1,y)\, dy=0.
\end{align*}
{Further, the compatibility condition for \eqref{eq:aux-2} is satisfied thanks to \eqref{eq:m'}.}
{Due to the elliptic regularity, $\partial_{x_1} \Psi(x_1, \cdot) \in C^{0,\alpha}(\bar Y)\cap H^1(Y)$.}
To prove the H\"older continuity in $x_1$, we estimate the $H^1$-norm of the difference $\Psi(\xi,\cdot)-\Psi(\eta,\cdot)$.
\begin{align*}
L_\xi(\Psi_\xi-\Psi_\eta) &= (L_\eta-L_\xi)\Psi_\eta + \mu_\xi \rho_\xi \Psi_\xi - \mu_\eta\rho_\eta\Psi_\eta,\\
(L_\xi(\Psi_\xi-\Psi_\eta), (\Psi_\xi-\Psi_\eta))_{L^2(Y)} &= ((L_\eta-L_\xi)\Psi_\eta, (\Psi_\xi-\Psi_\eta))_{L^2(Y)}\\
&+ \mu_\xi(\rho_\xi \Psi_\xi, (\Psi_\xi-\Psi_\eta))_{L^2(Y)}\\
&- \mu_\eta(\rho_\eta \Psi_\eta, (\Psi_\xi-\Psi_\eta))_{L^2(Y)}.
\end{align*}
{By the hypotheses (H1)--(H3), the differentiability of the eigenpair proved above, and Lemma \ref{lm:estimate}, we obtain
\begin{align}
\|\Psi_\xi-\Psi_\eta\|_{H^1(Y)}\le C|\xi-\eta|^\alpha.
\end{align}
Combining \eqref{eq:m'} and the above estimate yields that $\mu'(x_1)$ is H\"older continuous. For a given $\mu(x_1)$, $\partial_{x_1} \Psi(x_1,y)$ solves \eqref{eq:aux-2} and thus, by elliptic regularity, belongs to $H^1(Y)\cup C^{0,\alpha}(\bar Y)$. We conclude that $\Psi \in C^{1,\alpha}([-1,1]; C^{0,\alpha}(\bar Y)\cap H^1(Y))$.}

\end{proof}

\section{Proof of Theorem \ref{th:main}}
\label{sec: proof main theorem}
In this section we prove the main result of the paper Theorem \ref{th:main}. First, we perform the factorization of \eqref{orig-prob} with the positive eigenfunction defined in Lemma \ref{lm:existence-mu}. After that we rescale the obtained problem in order to elliminate the singularity. Then the two-scale convergence technique in spaces with singular measure is used to derive the limit problem \eqref{eq:eff-prob}.  
\subsection{Factorization}
\label{sec factorization and estimates eigenvalue}
As before, we assume that $\langle \rho(x_1, \cdot)\rangle=\int_Y \rho(x_1, y)\, dy<0$. 
Using the factorization $u^{\ve}(x) = \Psi(x_1, \frac{x}{\ve}) v^\ve(x)$, multiplying \eqref{orig-prob} by $\Psi(x_1, \frac{x}{\ve})$, we obtain the following spectral problem for the new unknowns $(\nu^\ve, v^\ve)$:
\begin{align}
\label{eq:v-eps}
\begin{cases}
\displaystyle
-{\rm div}(a_\Psi^\ve \nabla v^\ve)
+ \frac{1}{\ve}C^\ve v^\ve 
+ \frac{\mu(x_1)-\mu(0)}{\ve^2} \rho_\Psi^\ve v^\ve = \nu^\ve \rho_\Psi^\ve v^\ve, \quad x\in \Omega_\ve,\\[2mm]
\displaystyle
a_\Psi^\ve \nabla v^\ve \cdot \nu
{+  (a^\ve \Psi^\ve \nabla_x \Psi(x_1, y)\big|_{y=x/\ve} \cdot n)\, v^\ve} 
=0, \quad x\in \Sigma_\ve,\\[2mm]
v^{\ve}(-1, x')=v^{\ve}(1, x')=0, \quad x'\in \ve Q,
\end{cases}
\end{align}
where we have denoted
\begin{align}
\label{eq:C-eps}
\nu^\ve &= \lambda^\ve - \frac{\mu(0)}{\ve^2},\\
\label{def:a^eps rho^eps}
a_\Psi^\ve(x) &= a(x_1, y)\, \Psi(x_1, y)^2\big|_{y=\frac{x}{\ve}},\quad \rho_\Psi^\ve = \rho(x_1, y)\, \Psi(x_1, y)^2\big|_{y=\frac{x}{\ve}},\\
\label{def:C^eps}
C^\ve(x)
&= -  \ve \Psi(x_1, y){\rm div}(a(x_1,y) \nabla_x \Psi(x_1,y))\big|_{y=\frac{x}{\ve}}\\
&-\Psi(x_1,y) {\rm div}_x(a(x_1,y) \nabla_y \Psi(x_1,y))\big|_{y=\frac{x}{\ve}}.\nonumber
\end{align}
{Note that the gradient $\nabla_x\Psi$ has only one non-trivial component, the first one $\partial_{x_1} \Psi$. If $a(x_1, y)$ is a scalar function, the second term in the lateral boundary condition in \eqref{eq:v-eps} vanishes.}
{The weak form of \eqref{eq:v-eps} reads
\begin{align}
\label{eq:weak-v_eps}
A_\Psi(v^\ve, \varphi)&\equiv\int_{\Omega_\ve} a_\Psi^\ve \nabla v^\ve \cdot \nabla \varphi\, dx \nonumber\\
&+\int_{\Omega_\ve} a(x_1,y)\nabla_x\Psi(x_1,y)\big|_{y=x/\ve} \cdot \nabla(\Psi^\ve v^\ve \varphi)\, dx \nonumber\\
&- \frac{1}{\ve}\int_{\Omega_\ve} {\rm div}_x(a(x_1,y) \nabla_y \Psi(x_1,y)) \Psi(x_1, y)\big|_{y=x/\ve}\, v^\ve\, \varphi(x)\, dx\\
&+ \int_{\Omega_\ve}\frac{\mu(x_1)-\mu(0)}{\ve^2} \rho_\Psi^\ve v^\ve\,\varphi dx= 
\nu^\ve \int_{\Omega_\ve} \rho_\Psi^\ve v^\ve \varphi\, dx,\nonumber
\end{align}
for any $\varphi\in H^1(\Omega_\ve)$, $\varphi(\pm 1, x')=0$.}

{By the minmax principle, the first positive eigenvalue of \eqref{eq:v-eps} is given by
\begin{align}
\label{eq:nu-variational}
\nu_1^{\ve, +}
= \min_{{(\rho_\Psi^\ve v, v)=1}}A_\Psi(v, v),
\end{align}
where the bilinear form $A_\Psi(v, \varphi)$ is defined in \eqref{eq:weak-v_eps}, and the minimum is taken over functions $v\in H^1(\Omega_\ve)$ such that $v(\pm 1, x')=0$.
To minimize $A_\Psi(v,v)$, one can see that $v$ should be localized, that is $v$ may be of the form $v(\frac{x_{1}}{\ve^\gamma})$ for $\gamma>0$, as shown in the proof below.

\begin{lemma}
Under the assumptions (\textbf{H1})-(\textbf{H6}), the eigenvalues $\nu_j^\ve$ of \eqref{eq:weak-v_eps} satisfy the following estimate:
\begin{align}
\label{eq:est-nu^eps}
    |\nu_j^\ve| =|\lambda_j^{\ve, +} - \frac{\mu(0)}{\ve^2}| \le \frac{C}{\ve}, \quad j\in \mathbf N.
\end{align}
\end{lemma}
\begin{proof}
We begin by estimating the first positive eigenvalue $\nu_{1}^{\ve,+}$. For any $v\in H^1(\Omega_\ve)$, $(\rho_\Psi^\ve v,v)_{L^2(\Omega)}>0$ we have:
\begin{equation}
\label{eq:estimate above}
\begin{aligned}
    \nu_{1}^{\ve,+} &\leq
    \frac{1}{(\rho_{\Psi}^{\ve} v, v)_{L^{2}(\Omega_{\ve})}} 
     \Bigg( \int_{\Omega_\ve} a_\Psi^\ve \nabla v \cdot \nabla v \, dx\\
&+\int_{\Omega_\ve} a(x_1,y)\nabla_x\Psi(x_1,y)\big|_{y=x/\ve} \cdot \nabla(\Psi^\ve v^2 )\, dx\\
&- \frac{1}{\ve}\int_{\Omega_\ve} {\rm div}_x(a(x_1,y) \nabla_y \Psi(x_1,y)) \Psi(x_1, y)\big|_{y=x/\ve}\, v^2 \, dx \\
&+ \int_{\Omega_\ve}\frac{\mu(x_1)-\mu(0)}{\ve^2} \rho_\Psi^\ve v^2 \, dx \Bigg).
\end{aligned}
\end{equation}
Let us choose a test function $\displaystyle\varphi_{\ve} = \varphi(\frac{x_{1}}{\ve^{\gamma}})$, for some $0< \gamma <1$ and $\varphi \in C_c^\infty(\mathbb{R})$ with $\|\varphi\|_{L^2(\mathbb{R})} =1$. {Thus, $\|\varphi_\ve\|_{L^2(\Omega_\ve)}^2=O(\ve^{(d-1)} \ve^\gamma)$ and $\|\nabla \varphi_\ve\|_{L^2(\Omega_\ve)}^2=O(\ve^{(d-1)} \ve^{-\gamma})$.} With the help of (\textbf{H1})-(\textbf{H3}), \textbf{(H5)}, and by the regularity properties of $\Psi$, we have for small enough $\ve$:
\begin{align*}
    &\int_{\Omega_\ve} a_\Psi^\ve(x_{1}, \frac{x}{\ve}) \nabla \varphi(\frac{x_{1}}{\ve^{\gamma}}) \cdot \nabla \varphi(\frac{x_{1}}{\ve^{\gamma}})\, dx = O(\ve^{d-1-\gamma}),\\
    &\int_{\Omega_\ve} a(x_1,y)\nabla_x\Psi(x_1,y)\big|_{y=x/\ve} \cdot \nabla(\Psi^\ve \varphi(\frac{x_{1}}{\ve^{\gamma}})^2 )\, dx = O(\ve^{\gamma-1+d-1}),\\
    &-\frac{1}{\ve}\int_{\Omega_\ve} {\rm div}_x(a(x_1,y) \nabla_y \Psi(x_1,y)) \Psi(x_1, y)\big|_{y=x/\ve}\, \varphi(\frac{x_{1}}{\ve^{\gamma}})^2 \, dx = O(\ve^{\gamma-1+d-1}),\\
    &\int_{\Omega_\ve}\frac{\mu(x_1)-\mu(0)}{\ve^2} \rho_\Psi^\ve \varphi(\frac{x_{1}}{\ve^{\gamma}})^2\, dx = O(\ve^{d-1+3 \gamma -2}).
\end{align*}
{Using the normalization condition $\int_Y \rho \Psi^2\, dy=1$ and Lemma \ref{lm:estimate}, we obtain
\begin{align*}
    &\int_{\Omega_\ve} \rho_\Psi^\ve \varphi_\ve^2\, dx
    = \frac{1}{|Q|} \int_{\Omega_\ve} \big(\int_Y \rho \Psi^2\, dy\big) \varphi_\ve^2\, dx + O(\ve \|\varphi_\ve\|_{L^2(\Omega_\ve)})\|\nabla \varphi_\ve\|_{L^2(\Omega_\ve)}\\
    &=
\|\varphi_\ve\|_{L^2(\Omega_\ve)}^2 + O(\ve \|\varphi_\ve\|_{L^2(\Omega_\ve)}\|\nabla \varphi_\ve\|_{L^2(\Omega_\ve)})=O(\ve^\gamma \ve^{d-1}).
\end{align*}}
Then the estimate \eqref{eq:estimate above} becomes
\begin{align*}
    |\nu_{1}^{\ve,+}| \leq C\ve^{-\gamma}(\ve^{-\gamma} + \ve^{\gamma - 1} + \ve^{3 \gamma - 2}).
\end{align*}
Equating the powers of $\ve$ on the right-hand side of the last inequality we can see that  
the best choice of $\gamma$ for the considered type of test function is $\gamma = \frac{1}{2}$, and the estimate \eqref{eq:est-nu^eps} is proved for $j=1$.
The following eigenvalues $\nu_{j}^{\ve,+}$, $j = 2,3, \ldots$ can be estimated in a similar way by  choosing a test function that concentrates in a vicinity of $x_{1} = 0$, which is the minimum point of $\mu(x_{1})$, and is orthogonal to the first $j-1$ eigenfunctions $u_{k}^{\ve}$, $k = 1, \ldots, j-1$.

\end{proof}
\subsection{Rescaling}
\label{sec rescaling and apriori estimate}
We change the variables in \eqref{eq:v-eps} setting
$\displaystyle z= \frac{x}{\sqrt{\ve}}$.
With $v^{\ve}(\sqrt{\ve}z) = w^{\ve}(z)$, \eqref{eq:v-eps} becomes
\begin{align}
  \label{eq:rescle-final}
\begin{split}
\begin{cases}
     &\displaystyle 
     -\mathrm{div}(a_{\Psi}^{\ve}(\sqrt{\ve} z)\nabla w^{\ve}(z)) + C_\Psi^{\ve}w^{\ve}(z)\\[2mm]
     \displaystyle
     &\displaystyle+ \frac{\mu(\sqrt{\ve} z_{1}) - \mu(0)}{\ve} \rho_{\Psi}^{\ve}(\sqrt{\ve} z) w^{\ve}(z)
     = \ve \nu^{\ve} \rho_{\Psi}^{\ve}(\sqrt{\ve} z) w^{\ve}(z), \quad z\in \ve^{-\frac{1}{2}}\Omega_\ve,\\[2mm]
  &a_\Psi^\ve(\sqrt{\ve}z)\nabla w^\ve(z) \cdot \nu 
  + w^\ve(z) a^\ve \Psi^\ve \nabla_z \Psi(\sqrt{\ve}z_{1}, y)\big|_{y=z/\sqrt{\ve}} \cdot n =0, \quad \ve^{-\frac{1}{2}}\Sigma_\ve,\\[2mm]
      &w^{\ve}(-\ve^{-1/2},z')=w^{\ve}(\ve^{-1/2},z')=0, \quad z'\in \sqrt{\ve} Q,
      \end{cases}
    \end{split}
\end{align}
where 
\begin{align*}
    &a_{\Psi}^{\ve} (\sqrt{\ve}z) = a_{\Psi} \left( \sqrt{\ve}z_{1}, \frac{z}{\sqrt{\ve}} \right), \quad \rho_{\Psi}^{\ve}(\sqrt{\ve}z) = \rho_{\Psi} \left( \sqrt{\ve}z_{1}, \frac{z}{\sqrt{\ve}} \right),\\
    & 
    C_\Psi^{\ve}(\sqrt \ve z) = \left(- \Psi^{\ve} \mathrm{div}(a  \nabla_{z} \Psi) - \frac{1}{\sqrt{\ve}}\Psi^{\ve} \mathrm{div}_{z}(a \nabla_{\zeta} \Psi)\right)\left( \sqrt{\ve}z_{1}, \frac{z}{\sqrt{\ve}} \right).
\end{align*}
We will derive a priori estimates for the eigenfunctions of \eqref{eq:rescle-final} in terms of a singular measure $d\mu_\ve$ charging the rescaled domain $\ve^{-1/2}\Omega_\ve$.
Let us define a Radon measure on $\mathbb{R}^{d}$ by
\begin{align}
\label{def:mu_eps}
    \mu_{\ve}(B) = \displaystyle \frac{\ve^{-(d-1)/2}}{|Q|}\int_{B} \chi_{\ve^{-\frac{1}{2}} \Omega_{\ve}} \, dx,
\end{align}
for all Borel sets $B$, where $\chi_{\ve^{-\frac{1}{2}} \Omega_{\ve}}$ is the characteristic function on the rescaled domain $\ve^{-\frac{1}{2}} \Omega_{\ve}$. 

Let us choose the normalization in the original domain $\Omega_{\ve}$:
$\int_{\Omega_{\ve}} v_{i}^{\ve} v_{j}^{\ve} \, dx = \ve^{1/2} \ve^{d-1} |Q| \delta_{ij}$.
Then in the rescaled domain, the normalization condition becomes
\begin{align}
\label{eq:norm-cond-rescaled}
    \int_{\ve^{-1/2}\Omega_{\ve}} v_{i}^{\ve}(\sqrt{\ve} z) v_{j}^{\ve}(\sqrt{\ve}z) \, dz = \ve^{(d-1)/2} |Q| \delta_{ij},
\end{align}
and in terms of the singular measure:
\begin{align}
\label{eq:norm-cond-rescaled}
    \int_{\mathbb{R}^{d}} v_{i}^{\ve} v_{j}^{\ve} \, d\mu_{\ve} = \delta_{ij}.
\end{align}

\begin{lemma}
\label{lm:a priori est w_eps}
Suppose that \textbf{(H1)}--\textbf{(H6)} are satisfied. Let $(\nu^{\ve},w^{\ve})$ be an eigenpair of \eqref{eq:rescle-final} with the eigenfunctions normalized by $\int_{\mathbb R^d} w_i^\ve w_j^\ve\, d\mu_\ve = \delta_{ij}$.
Then $w^\ve$ satisfies the following estimate
\begin{align}
\label{eq:estimate}
    \int_{\mathbb R^d} |\nabla w_\ve|^2\, d\mu_\ve
    + \int_{\mathbb R^d} |z_1 w_\ve|^2\, d\mu_\ve \le C,
\end{align}
with a constant $C$ independent of $\ve$.
\end{lemma}
\begin{proof}
Weak formulation of the rescaled problem \eqref{eq:rescle-final} reads
\begin{align}
\label{eq:weak-w_eps}
&\int_{\ve^{-1/2}\Omega_\ve} a_\Psi^\ve(\sqrt{\ve}z) \nabla w^\ve(z) \cdot \nabla \varphi(z)\, dz \nonumber\\
&+\sqrt{\ve}\int_{\ve^{-1/2}\Omega_\ve} a(x_1,\zeta)\nabla_x\Psi(x_1,\zeta)\big|_{x_{1}=\sqrt{\ve}z_{1}, \, \zeta =z/\sqrt{\ve}} \cdot \nabla(\Psi^\ve(\sqrt{\ve}z))\, w^\ve(z) \varphi(z)\, dz \nonumber\\
&+\sqrt{\ve}\int_{\ve^{-1/2}\Omega_\ve} a(x_1,\zeta)\nabla_x\Psi(x_1,\zeta)\big|_{x_{1}=\sqrt{\ve}z_{1}, \, \zeta =z/\sqrt{\ve}} \cdot \nabla w^\ve(z) \varphi(z) \Psi^\ve(\sqrt{\ve}z) \, dz \nonumber\\
&+\sqrt{\ve}\int_{\ve^{-1/2}\Omega_\ve} a(x_1,\zeta)\nabla_x\Psi(x_1,\zeta)\big|_{x_{1}=\sqrt{\ve}z_{1}, \, \zeta =z/\sqrt{\ve}} \cdot w^\ve(z) \nabla \varphi(z)\Psi^\ve(\sqrt{\ve}z) \, dz \nonumber\\
&- \int_{\ve^{-1/2}\Omega_\ve} {\rm div}_x(a(x_1,\zeta) \nabla_\zeta \Psi(x_1,\zeta)) \Psi(x_1, \zeta)\big|_{x_{1}=\sqrt{\ve}z_{1}, \, \zeta =z/\sqrt{\ve}}\, w^\ve(z)\, \varphi(z)\, dz\\
&+ \int_{\ve^{-1/2}\Omega_\ve}\frac{\mu(\sqrt{\ve}z_1)-\mu(0)}{\ve} \rho_\Psi^\ve(\sqrt{\ve}z) w^\ve(z)\,\varphi(z) dz\\
&= 
\ve \nu^\ve \int_{\ve^{-1/2}\Omega_\ve} \rho_\Psi^\ve(\sqrt{\ve}z) w^\ve(z) \varphi(z)\, dz,\nonumber
\end{align}
for all $\phi \in H^{1}(\ve^{-1/2}\Omega_\ve)$ such that $\phi(-1/\sqrt \ve, z')=\phi(1/\sqrt \ve, z')=0$.
{
We use $w^\ve$ as a test function in the weak formulation above and rewrite it in terms of $\mu_\ve$:
\begin{align}
\label{eq:weak-w_eps-measure}
&\int_{\mathbb{R}^{d}} a_\Psi^\ve(\sqrt{\ve}z) \nabla w^\ve \cdot \nabla w^\ve\, d\mu_\ve \nonumber\\
&+\sqrt{\ve}\int_{\mathbb{R}^{d}} a(x_1,\zeta)\nabla_x\Psi(x_{1},\zeta)\big|_{x_{1} = \sqrt{\ve}z_{1}, \, \zeta = z/\sqrt{\ve}} \cdot \nabla\Psi(\sqrt{\ve}z)\, (w^\ve)^2\, d\mu_\ve \nonumber\\
&+\sqrt{\ve}\int_{\mathbb{R}^{d}} a(x_1,\zeta)\nabla_x\Psi(x_{1}, \zeta)\big|_{x_{1} = \sqrt{\ve}z_{1}, \, \zeta = z/\sqrt{\ve}} \cdot \nabla w^\ve\, w^\ve \Psi(\sqrt{\ve}z) \, d\mu_\ve \nonumber\\
&+\sqrt{\ve}\int_{\mathbb{R}^{d}} a(x_1,\zeta)\nabla_x\Psi(x_{1}, \zeta)\big|_{x_{1} = \sqrt{\ve}z_{1}, \, \zeta = z/\sqrt{\ve}} \cdot w^\ve(z) \nabla w^\ve \Psi(\sqrt{\ve}z) \, d\mu_\ve \nonumber\\
&- \int_{\mathbb{R}^{d}} {\rm div}_x(a(x_1,\zeta) \nabla_\zeta \Psi(x_1,\zeta)) \Psi(x_1, \zeta)\big|_{x_{1}=\sqrt{\ve}z_{1}, \, \zeta = z/\sqrt{\ve}}\, (w^\ve)^2\, d\mu_\ve\\
&+ \int_{\mathbb{R}^{d}}\frac{\mu(\sqrt{\ve}z_1)-\mu(0)}{\ve} \rho_\Psi^\ve(\sqrt{\ve}z) (w^\ve)^2 d\mu_\ve\\
&= 
\ve \nu^\ve \int_{\ve^{-1/2}\Omega_\ve} \rho_\Psi^\ve(\sqrt{\ve}z) (w^\ve)^2\, d\mu_\ve.\nonumber
\end{align}
The right-hand side of the above identity is estimated using Lemma \ref{lm:oscillating-integrals}, the normalization in Lemma \ref{lm:existence-mu}, the estimate for the eigenvalues \eqref{eq:est-nu^eps}, and the normalization condition for $w_i^\ve$:
\begin{align}
    \label{eq:est:RHS}
    \ve \nu^\ve \left|\int_{\ve^{-1/2}\Omega_\ve} \rho_\Psi^\ve(\sqrt{\ve}z) |w^\ve|^2\, dz \right|
    \le C.
\end{align}
Let us estimate the second term in the equality above separately, using the regularity properties of the coefficients and $\Psi(x_1, y)$:
\begin{align*}
& \left|\sqrt{\ve}\int_{\ve^{-1/2}\Omega_\ve} a(x_1,\zeta)\nabla_x\Psi(x_1,\zeta)\big|_{x_{1}=\sqrt{\ve}z_{1}, \, \zeta =z/\sqrt{\ve}} \cdot \nabla(\Psi^\ve(\sqrt{\ve}z))\, (w^\ve(z))^2\, dz \right|\\
&\le
\sqrt{\ve} \left|\int_{\ve^{-1/2}\Omega_\ve} a(x_1,\zeta)\nabla_x\Psi(x_1,\zeta)\big|_{x_{1}=\sqrt{\ve}z_{1}, \, \zeta =z/\sqrt{\ve}} \cdot \nabla_z(\Psi^\ve(\sqrt{\ve}z))\, (w^\ve(z))^2\, dz \right| \\
&+ \left|\int_{\ve^{-1/2}\Omega_\ve} a(x_1,\zeta)\nabla_x\Psi(x_1,\zeta)\big|_{x_{1}=\sqrt{\ve}z_{1}, \, \zeta
 =z/\sqrt{\ve}} \cdot \nabla_\zeta \Psi^\ve(\sqrt{\ve}z)\, (w^\ve(z))^2\, dz \right|\\
&\le 
C_1\ve \int_{\ve^{-1/2}\Omega_\ve} |w^\ve|^2\, dz
+ C_2\int_{\ve^{-1/2}\Omega_\ve}|w^\ve|^2\, dz\le C_0 \int_{\ve^{-1/2}\Omega_\ve}|w^\ve|^2\, dz.
\end{align*}
This shows that this term does not vanish, as $\ve \to 0$, and, thus, we shift the spectrum by adding $C_0 \|w^\ve\|_{L^2(\ve^{-1/2}\Omega_\ve)}^2$ to both sides of the weak formulation for the rescaled equation. 

As $\rho_{\Psi}^{\ve}(\sqrt{\ve}z)$ is sign-changing, we cannot estimate the term containing $(\mu(\sqrt{\ve} z_1) - \mu(0))$ directly from below. Then we will use the following mean-value theorem to estimate this term, using Corollary \eqref{corollary_two scale_average}
{\begin{align*}
    &\int_{\mathbb{R}^{d}} 
    \frac{(\mu(\sqrt{\ve}z_{1}) - \mu(0))}{\ve} \rho_{\Psi}^{\ve}(\sqrt{\ve}z) (w^{\ve}(z))^{2} \, d \mu_{\ve}\\
    &~~~~~ -\int_{\mathbb{R}^{d}} \frac{(\mu(\sqrt{\ve}z_{1}) - \mu(0))}{\ve} \Big( \int_{Y} \rho_{\Psi}^{\ve}(\sqrt{\ve}z_{1},\zeta) \, d\zeta \Big) (w^{\ve}(z))^{2} \, d \mu_{\ve}\\
    &= O(\sqrt{\ve} \|w^{\ve}\|_{L^{2}(\mathbb{R}^{d}, \mu_{\ve})} \| \nabla w^{\ve}\|_{L^{2}(\mathbb{R}^{d}, \mu_{\ve})}).
\end{align*}}
By \textbf{(H5)},
\begin{align*}
    \mu(\sqrt{\ve}z_{1}) - \mu(0) = \frac{1}{2} \mu''(0)(\sqrt{\ve}|z_{1}|)^{2} + o(|\sqrt{\ve}z_{1}|^{2}).
\end{align*}
Since we do not know yet that $w^{\ve}$ is localized, Taylor's expansion cannot be used to obtain an estimate for the remainder. Instead, we will use a quadratic equivalence, a forward consequence to Taylor's theorem. 
We substitute $(\mu(\sqrt{\ve}z_{1}) - \mu(0)) \Big( \int_{Y} \rho_{\Psi}^{\ve}(\sqrt{\ve}z_{1},\zeta) \, d\zeta \Big)$ with the equivalent quadratic function $\frac{\mu''(0)}{2} (\sqrt{\ve} |z_{1}|)^{2}$ in the weak formulation \eqref{eq:weak-w_eps-measure}.

Finally, by the coercivity of $a$ and the regularity properties of $\Psi$, we derive \eqref{eq:estimate}. Note that the estimate {$\|z_{1} w_{\ve}\|_{L^{2}(\mathbb{R}^{d}, \mu_{\ve})} \leq C$ implies $\|w_{\ve}\|_{L^{2}(\mathbb{R}^{d}, \mu_{\ve})} \leq C$ because of the growing weight $|z_{1}|$.} }
\end{proof}
The proof of the following lemma can be found in \cite{pettersson2017two}, \cite{PaPe-ApplicAnal-2014}.
\begin{lemma} 
Under conditions in Lemma \ref{lm:a priori est w_eps}, there exists $w \in H^{1}(\mathbb{R}^{d},\mu^{*})$ such that
\begin{align*}
&w^{\ve}\overset{2}{\rightharpoonup} w(z_{1},0) \quad \text{in } L^{2}(\mathbb{R}^{d}, \mu_{\ve}),\\
&\nabla w^{\ve} \overset{2}{\rightharpoonup}
\nabla^{\mu^{*}}w(z_{1},0) + \nabla_{\zeta} w^{1}(z_{1},\zeta) \quad \text{in } L^{2}(\mathbb{R}^{d},\mu_{\ve}),
\end{align*}
where $w^{1}(z_{1},\zeta) \in L^{2}(\mathbb{R}; H^{1}(Y))$ is $1$-periodic in $\zeta_{1}$.
\end{lemma}
 
\subsection{Passage to the limit in \eqref{eq:weak-w_eps}}
\label{sec passage of limit}

Recall the weak formulation of the rescaled problem \eqref{eq:rescle-final}:
\begin{align}
\label{eq:weak-w_eps-measure-1}
&\int_{\mathbb{R}^{d}} a_\Psi^\ve(\sqrt{\ve}z) \nabla w^\ve(z) \cdot \nabla \varphi(z)\, d\mu_\ve \nonumber\\
&+\sqrt{\ve}\int_{\mathbb{R}^{d}} a(x_1,\zeta)\nabla_x\Psi(x_1, \zeta )\big|_{x_{1}=\sqrt{\ve}z_{1}, \, \zeta =z/\sqrt{\ve}} \cdot \nabla\Psi^\ve(\sqrt{\ve}z)\, w^\ve(z) \varphi(z)\, d\mu_\ve \nonumber\\
&+\sqrt{\ve}\int_{\mathbb{R}^{d}} a(x_1,\zeta)\nabla_x\Psi(x_1,\zeta)\big|_{x_{1}=\sqrt{\ve}z_{1}, \, \zeta =z/\sqrt{\ve}} \cdot \nabla w^\ve(z) \varphi(z) \Psi^\ve(\sqrt{\ve}z) \, d\mu_\ve \nonumber\\
&+\sqrt{\ve}\int_{\mathbb{R}^{d}} a(x_1,\zeta)\nabla_x\Psi(x_1,\zeta)\big|_{x_{1}=\sqrt{\ve}z_{1}, \, \zeta =z/\sqrt{\ve}} \cdot w^\ve(z) \nabla \varphi(z)\Psi^\ve(\sqrt{\ve}z) \, d\mu_\ve \nonumber\\
&- \int_{\mathbb{R}^{d}} {\rm div}_x(a(x_1,\zeta) \nabla_\zeta
\Psi(x_1,\zeta)) \Psi(x_1, \zeta)\big|_{x_{1}=\sqrt{\ve}z_{1}, \, \zeta =z/\sqrt{\ve}}\, w^\ve(z)\, \varphi(z)\, d\mu_\ve\\
&+ \int_{\mathbb{R}^{d}}\frac{\mu(\sqrt{\ve}z_1)-\mu(0)}{\ve} \rho_\Psi^\ve(\sqrt{\ve}z) w^\ve(z)\,\varphi(z) d\mu_\ve \nonumber\\
&= 
\ve \nu^\ve \int_{\mathbb{R}^{d}} \rho_\Psi^\ve(\sqrt{\ve}z) w^\ve(z) \varphi(z)\, d\mu_\ve,\nonumber
\end{align}
\noindent
for all $\phi \in H^{1}(\ve^{-1/2}\Omega_\ve)$ such that $\phi(-1/\sqrt \ve, z')=\phi(1/\sqrt \ve, z')=0$.

\noindent
\textbf{Step 1.}\\
Choose a test function $\Phi_{\ve} = \ve^{1/2} \phi(z) \psi(\frac{z}{\ve^{1/2}})$ in \eqref{eq:weak-w_eps-measure-1}, where $\phi \in C_{0}^{\infty}(\mathbb{R}^{d})$, $\psi \in C^{\infty}(\overline{Y})$.
The gradient of the test function is
\begin{align*}
    \nabla \Phi_{\ve} = \ve^{1/2} \psi(\frac{z}{\ve^{1/2}}) \nabla \phi(z) + \phi(z) \nabla_{\zeta} \psi(\zeta)\Big|_{\zeta = z/\ve^{1/2}}.
\end{align*}
The limit of the first term in \eqref{eq:weak-w_eps} is
\begin{align*}
    &\lim_{\ve \rightarrow 0} \int_{\mathbb{R}^{d}} a_{\Psi}^{\ve}(\sqrt{\ve}z) \nabla w^{\ve}(z) \cdot \nabla \Phi^{\ve}(z) \, d\mu_{\ve}\\
    &= \frac{1}{|Y|} \int_{\mathbb{R}^{d}} \left( \int_{Y} a_{\Psi}(0,\zeta)\left (\nabla^{\mu^{*}}w(z_{1},0) + \nabla_{\zeta} w^{1}(z_{1},\zeta) \right) \cdot \nabla_{\zeta} \psi(\zeta) \, d\zeta\right) \phi(z_{1},0) \, d\mu^{*}.
\end{align*}
{The limit of the next four terms is zero} due to the regularity properties of $\Psi$ and the small factor $\sqrt \ve$ in the test function.
The sixth term can be proved to go to zero by using Corollary 1 in \cite{PaPe-ApplicAnal-2014}. Indeed,
\begin{align*}
    &\int_{\mathbb{R}^{d}} \frac{\mu(\sqrt{\ve}z_{1}) - \mu(0)}{\ve} \rho_{\Psi}^{\ve}(\sqrt{\ve}z) w^{\ve} \Phi^{\ve}(z) \, d\mu_{\ve}\\
    &\\
    &= \displaystyle  \frac{\sqrt{\ve}}{|Y|}\int_{\mathbb{R}^{d}} \frac{\mu(\sqrt{\ve}z_{1}) - \mu(0)}{\ve} \left( \int_{Y} \rho_{\Psi}(\sqrt{\ve} z_{1}, \zeta) \psi(\zeta) \, d\zeta \right) w^{\ve} \phi(z) \, d\mu_{\ve} + O(\sqrt{\ve})\\
    &\\
    &= \displaystyle \frac{\sqrt{\ve} \mu''(0)}{2|Y|} \int_{\mathbb{R}^{d}} |z_{1}|^{2} \left( \int_{Y} \rho_{\Psi}(\sqrt{\ve} z_{1}, \zeta) \psi(\zeta) \, d\zeta \right) w^{\ve} \phi(z) \, d\mu_{\ve} \\[2mm]
    &~~~~~~~~~~~~~~~+  \frac{\sqrt{\ve}}{|Y|} \int_{\mathbb{R}^{d}} \frac{o(\ve |z_{1}|^{2})}{2\ve} \left( \int_{Y} \rho_{\Psi}(\sqrt{\ve} z_{1}, \zeta) \psi(\zeta) \, d\zeta \right) w^{\ve} \phi(z) \, d\mu_{\ve} + O(\sqrt{\ve}).
\end{align*}
Since $\phi$ has a compact support in $z_1$, $o(\ve |z_1|^2)=o(\ve)$ as $\ve \to 0$. Therefore
\begin{align*}
    &\left| \int_{\mathbb{R}^{d}} \frac{\mu(\sqrt{\ve}z_{1}) - \mu(0)}{\ve} \rho_{\Psi}^{\ve}(\sqrt{\ve}z) w^{\ve} \Phi^{\ve}(z) \, d\mu_{\ve} \right|\\
    &\leq \left| \frac{\sqrt{\ve} \mu''(0)}{2 |Y|}\int_{\mathbb{R}^{d}} |z_{1}|^{2} \left( \int_{Y} \rho_{\Psi}(\sqrt{\ve} z_{1}, \zeta) \psi(\zeta) \, d\zeta \right) w^{\ve} \phi(z) \, d\mu_{\ve}  + O(\sqrt{\ve}) \right| \le C \sqrt{\ve}.
\end{align*}
As for the right-hand side of the weak formulation \eqref{eq:weak-w_eps-measure-1}, it is estimated by using the bound for the eigenvalues \eqref{eq:est-nu^eps}:
\begin{align*}
    &\left| \ve \nu^{\ve} \int_{\mathbb{R}^{d}} \rho_{\Psi}^{\ve}(\sqrt{\ve}z) w^{\ve} \Phi^{\ve}(z) \, d\mu_{\ve} \right|
    \le C \sqrt{\ve}.
\end{align*}
Passing limit as $\ve \rightarrow 0$ in \eqref{eq:weak-w_eps-measure-1}, we obtain a problem for $w^1$:
\begin{align}
\label{limit_rescle_Phi^(ve)}
    \int_{\mathbb{R}^{d}} \left( \int_{Y} a_{\Psi}(0,\zeta)\left (\nabla^{\mu^{*}}w(z_{1},0) + \nabla_{\zeta} w^{1}(z_{1},\zeta) \right) \cdot \nabla_{\zeta} \psi(\zeta) \, d\zeta\right) \phi(z_{1},0) \, d\mu^{*} = 0.
\end{align}
We are looking for the solution in the form $w^{1}(z_{1}, \zeta) = N(\zeta) \cdot \nabla^{\mu^{*}} w(z_{1},0)$, where $N: Y \to \mathbb{R}^d$ are periodic in $\zeta_1$ solving
\begin{align*}
    &\int_{Y} a_{\Psi}(0,\zeta) \nabla_{\zeta} N_{j}(\zeta) \cdot \nabla_{\zeta} \psi(\zeta) \, d\zeta = - \int_{Y} (a_{\Psi})_{kj} \partial_{\zeta_{k}} \psi(\zeta)) d\zeta.
\end{align*}
The last integral identity yields the problem for $N_j$ in its strong form:
\begin{align}
\label{rescle_cell_problem}
    \begin{cases}
        &-\displaystyle \mathrm{div}(a_{\Psi}(0,\zeta) \nabla_{\zeta} N_{j}(\zeta)) = \partial_{\zeta_{k}} (a_{\Psi})_{kj} (0,\zeta) \quad \text{in } Y\\
        &a_{\Psi}(0,\zeta) \nabla_{\zeta} N_{j}(\zeta) \cdot n = - (a_{\Psi})_{kj} (0, \zeta) n_{k} \quad \text{on } \Sigma, \\
        & N_{j}(\cdot, \zeta')\,\, \mbox{is}\,\, 1-\mbox{periodic}, \quad j=1,2,\ldots, d.
    \end{cases}
\end{align}
In this way, we have the following convergence
\begin{align*}
    &\nabla w^{\ve}(z) \overset{2}{\rightharpoonup} \left( I + \nabla N(\zeta) \right) \nabla^{\mu^{*}} w(z_{1},0) ~~ \text{as } \ve \rightarrow 0.
    \end{align*}
 We proceed with deriving the effective problem for the limit function $w$.\\

\noindent
\textbf{Step 2:}
\noindent 
Let us take a test function $\varphi\in C_{0}^{\infty}(\mathbb{R}; C^{\infty}(\ve^{1/2}Q))$ in \eqref{eq:weak-w_eps-measure-1}
\begin{align}
\label{eq:weak-pass-limit}
\begin{split}
      &\int_{\mathbb{R}^{d}} a_{\Psi}^{\ve}(\sqrt{\ve}z) \nabla w^{\ve} \cdot \nabla \varphi \, d\mu_{\ve}  + \int_{\mathbb{R}^{d}}  (a^{\ve} \nabla_{z} \Psi^{\ve} \cdot \nabla(\Psi^{\ve} w^{\ve} \varphi) \, d\mu_{\ve}\\
      &- \int_{\mathbb{R}^{d}} \Psi^{\ve} \mathrm{div}_{z} (a^{\ve} \nabla_{\zeta} \Psi^{\ve} ) w^{\ve} \varphi(z) \, d\mu_{\ve}\\
    &+ \int_{\mathbb{R}^{d}} \frac{\mu(\sqrt{\ve}z_{1}) - \mu(0)}{\ve} \rho_{\Psi}^{\ve}(\sqrt{\ve}z) w^{\ve} \varphi(z) \, d\mu_{\ve}\\
    &~= \ve \nu^{\ve} \int_{\mathbb{R}^{d}} \rho_{\Psi}^{\ve}(\sqrt{\ve}z) w^{\ve} \varphi(z) \, d\mu_{\ve}.
\end{split}
\end{align}
The first term in \eqref{eq:weak-pass-limit}, as $\ve \to 0$,  yields
 \begin{align*}
     &\lim_{\ve \rightarrow 0} \int_{\mathbb{R}^{d}} a_{\Psi}^{\ve}(\sqrt{\ve}z) \nabla w^{\ve}(z) \cdot \nabla \varphi(z) \, d\mu_{\ve}\\
     &~~~= \frac{1}{|Y|} \int_{\mathbb{R}^{d}} \left( \int_{Y} a_{\Psi}(0,\zeta) \left( I + \nabla N(\zeta) \right) \, d\zeta\right) \nabla^{\mu^{*}} w(z_{1},0) \cdot \nabla \varphi(z_{1},0) \, d\mu^{*}\\
     &~~~= \int_{\mathbb{R}^{d}} A^{\text{eff}} \nabla^{\mu^{*}} w(z_{1},0) \cdot \nabla \varphi(z_{1},0) \, d\mu^{*},
 \end{align*}
where we denote
 \begin{align*}
     &A_{ij}^{\text{eff}} = \frac{1}{|Y|} \int_{Y}(a_{\Psi})_{ik}(0,\zeta) \left(  \delta_{kj} + \partial_{\zeta_{k}} N_{j}(\zeta) \right) \, d \zeta.
 \end{align*}
The fourth term in  \eqref{eq:weak-pass-limit}:
 \begin{align*}
     &\int_{\mathbb{R}^{d}} \frac{\mu(\sqrt{\ve}z_{1}) - \mu(0)}{\ve} \rho_{\Psi}^{\ve}(\sqrt{\ve}z) w^{\ve} \varphi(z) \, d\mu_{\ve}\\
     &= \frac{1}{|Y|} \int_{\mathbb{R}^{d}} \frac{\mu(\sqrt{\ve}z_{1}) - \mu(0)}{\ve} \left(\int_{Y} \rho_{\Psi}(0,\zeta) \, d\zeta\right) w^{\ve} \varphi(z) \, d\mu_{\ve}\\
     &+ O(\sqrt{\ve} \|w^{\ve}\|_{L^{2}(\mathbb{R}^{d},\mu_{\ve})} \|\nabla w^{\ve}\|_{L^{2}(\mathbb{R}^{d},\mu_{\ve})} )\\
     &= \frac{1}{|Y|} \int_{\mathbb{R}^{d}} \frac{\mu(\sqrt{\ve}z_{1}) - \mu(0)}{\ve} \left(\int_{Y} \rho_{\Psi}(\sqrt{\ve}z_{1},\zeta) \, d\zeta\right) w^{\ve} \varphi(z) \, d\mu_{\ve}+ O(\sqrt{\ve}).
 \end{align*}
  Writing the Taylor expansions for $\mu(\sqrt{\ve}z_{1})$ and $\langle \rho_{\Psi}(\sqrt{\ve}z_{1} , \cdot) \rangle$ we obtain, as $\ve|z_1|^2 \to 0$:
  \begin{align*}
      &\mu(\sqrt{\ve}z_{1}) = \mu(0) + \ve z_{1}^{2} \mu''(0) + o(\ve |z_{1}|^{2}),\\
      &\langle \rho_{\Psi}(\sqrt{\ve}z_{1}, \cdot) \rangle = \langle \rho_{\Psi}(0, \cdot) \rangle + \sqrt{\ve}z_{1} \langle \partial_{z_{1}} \rho_{\Psi}(0, \cdot) \rangle + \frac{\ve |z_{1}|^{2}}{2} \langle \partial_{z_{1}}^{2} \rho_{\Psi}(0, \cdot) \rangle + o(\ve |z_{1}|^{2}).
  \end{align*}
  By the Lebesgue dominated convergence theorem, since the test function has compact support,
  \begin{align*}
      &\lim_{\ve \rightarrow 0} \int_{\mathbb{R}^{d}} \frac{\mu(\sqrt{\ve}z_{1}) - \mu(0)}{\ve} \rho_{\Psi}^{\ve}(\sqrt{\ve}z) w^{\ve} \varphi(z) \, d\mu_{\ve}\\
      &= \frac{\mu''(0)}{2} \langle \rho_{\Psi}(0, \cdot) \rangle \int_{\mathbb{R}^{d}} |z_{1}|^{2} w(z_{1},0) \varphi(z_{1},0) \, d\mu^{*}.
   \end{align*}
As for the terms 2-4 in \eqref{eq:weak-w_eps-measure-1} coming from $C^\ve$, we see that the only one that contributes to the limit is the one containing $\nabla_\zeta \Psi$ (in the second term). 
\begin{align*}
    &\lim_{\ve \rightarrow 0} \sqrt{\ve} \int_{\mathbb{R}^{d}}  a(x_{1}, \zeta) \nabla_{x} \Psi(x_{1}, \zeta)\big|_{x_{1} = \sqrt{\ve}z_{1}, \, \zeta = z/\sqrt{\ve}} \cdot \frac{1}{\sqrt{\ve}} (\nabla_{\zeta} \Psi)(\sqrt{\ve}z)  \, w^{\ve}(z) \varphi(z)   \, d\mu_{\ve}\\
    &= \frac{1}{|Y|}\int_{\mathbb{R}^{d} \times Y} (a\nabla_{x} \Psi) (0, \zeta) \cdot \nabla_{\zeta} \Psi(0, \zeta)\, w(z_{1}) \varphi(z_{1},0) \, d\zeta\, d\mu^{*};\\
    &\\
    &\lim_{\ve \rightarrow 0} \int_{\mathbb{R}^{d}} \Psi(x_{1}, \zeta)\,  \mathrm{div}_{x} (a(x_{1}, \zeta) \nabla_{\zeta} \Psi(x_{1}, \zeta))\big|_{x_{1} = \sqrt{\ve}z_{1}, \, \zeta = z/\sqrt{\ve} }  \, w^{\ve}(z) \varphi(z) \, d\mu_{\ve} \\
    &= \frac{1}{|Y|}\int_{\mathbb{R}^{d} \times Y} \Psi (0, \zeta) \mathrm{div}_{x}(a \nabla_{\zeta} \Psi)(0, \zeta) \, w(z_{1}) \varphi(z_{1},0) \, d\zeta\, d\mu^{*}.
     \end{align*}
Denote    
\begin{align*}
        c^{\text{eff}} = \frac{1}{|Y|}\int_{Y}\Big( (a \nabla_{x} \Psi)(0,\zeta) \cdot \nabla_{\zeta}\Psi(0, \zeta) - \mathrm{div}_{x}(a \nabla_{\zeta} \Psi)(0, \zeta) \Psi(0, \zeta) \Big)\, d\zeta.
\end{align*}
The limit of the right-hand side of the weak formulation \eqref{eq:weak-w_eps-measure-1} is
    \begin{align*}
        &\displaystyle \lim_{\ve \rightarrow 0} \ve \nu^{\ve} \int_{\mathbb{R}^{d}} \rho_{\Psi}^{\ve} (\sqrt{\ve}z ) w^{\ve}(z) \varphi(z) \, d\mu_{\ve}
        = \nu \langle \rho_{\Psi}(0, \cdot) \rangle \int_{\mathbb{R}^{d}} w(z_{1},0) \varphi(z_{1}, 0) \, d\mu^{*}.
    \end{align*}
Finally, passing to the limit of the weak formulation \eqref{eq:weak-w_eps-measure-1} we have
\begin{equation}
\label{limit_rescaled_weak form}
\begin{aligned}
&\int_{\mathbb{R}^{d}} A^{\text{eff}} \nabla^{\mu^{*}} w(z_{1},0) \cdot \nabla \phi(z_{1},0) \, d\mu^{*}\\
&+ \frac{\mu''(0)}{2} \langle \rho_{\Psi}(0, \cdot) \rangle \int_{\mathbb{R}^{d}} |z_{1}|^{2} w(z_{1},0) \varphi(z_{1},0) \, d\mu^{*}\\
&+ \int_{\mathbb{R}^{d}} c^{\text{eff}} w(z_{1},0) \phi(z_{1},0) \, d\mu^{*} = \nu \langle \rho_{\Psi}(0, \cdot) \rangle \int_{\mathbb{R}^{d}} w(z_{1},0) \varphi(z_{1}, 0) \, d\mu^{*}.
\end{aligned}
\end{equation}
Take any test function with zero trace $\varphi(z_{1},0, \ldots, 0) = 0$ and a non-zero $\mu^{*}$-gradient, \textit{e.g.} $\varphi(z) = \displaystyle \sum_{j \neq 1} z_{j} \psi(z_{1})$, with arbitrary $\psi \in C_{c}^{\infty}(\mathbb{R})\setminus\{0\}$. Then
\begin{align*}
&\varphi(z) = z \cdot \Big(0, \psi_{2}(z_{1}), \ldots, \psi_{d}(z_{1}) \Big), ~~\varphi(z_{1}, 0, \ldots,0) = 0,\\ &\nabla\varphi(z) = \displaystyle \Big(\sum_{j \neq 1} z_{j}\psi_{j}'(z_{1}), \psi_{2}(z_{1}), \ldots, \psi_{d}(z_{1}) \Big),\\ 
&\nabla^{\mu^{*}} \varphi(z) = \nabla\varphi(z_{1}, 0, \ldots, 0) = \Big(0, \psi_{2}(z_{1}), \ldots, \psi_{d}(z_{1}) \Big).
\end{align*}
By the density of $C_{c}^{\infty}(\mathbb{R}^{d})$ in $L^{2}(\mathbb{R}^{d})$, we can take $\psi_{j} \in L^{2}(\mathbb{R}^{d})$. Taking this test function in \eqref{limit_rescaled_weak form} gives
\begin{align*}
    \int_{\mathbb{R}^{d}} A^{\text{eff}} \nabla^{\mu^{*}} w(z_{1},0) \cdot \Big(0, \psi_{2}(z_{1}), \ldots, \psi_{d}(z_{1}) \Big) \, d\mu^{*} = 0
\end{align*}
which implies
\begin{align*}
&A^{\text{eff}} \nabla^{\mu^{*}} w(z_{1},0) = \Big( \displaystyle \sum_{j=1}^{d} A_{1j}^{\text{eff}} \partial_{z_{j}}^{\mu^{*}} w(z_{1},0), 0, \ldots, 0 \Big).
\end{align*}
\noindent
Let us reformulate the equation for $N_{k}$ in the following form:
\begin{equation}
\label{rescaled_cell_problem}
\begin{aligned}
    -&\mathrm{div}(a_{\Psi}(0, \zeta) \nabla (N_{k}(\zeta) + \zeta_{k})) = 0, \quad y\in Y\\
    &a_{\Psi}(0, \zeta) \nabla (N_{k} + \zeta_{k}) = 0, \quad y\in \Sigma,\\
    &N_k(\cdot, \zeta') \,\, \mbox{is} \,\, 1-\mbox{periodic}.
\end{aligned}
\end{equation}
Multiplying \eqref{rescaled_cell_problem} by $\zeta_{m}$ for $m\neq 1$ ($\zeta_{m}$  is then $1$-periodic in $\zeta_{1}$ and thus it can be used as a test function), and integrating by parts over $Y$, we obtain
\begin{align*}
    \int_{Y} a_{\Psi}(0, \zeta) \nabla (N_{k}(\zeta) + \zeta_{k}) \cdot \nabla \zeta_{m} \, d\zeta = 0,
\end{align*}
which yields
\begin{align*}
&0= \int_{Y} \displaystyle \sum_{j = 1}^{d} (a_{\Psi})_{mj}(0, \zeta) (\partial_{\zeta_{j}}N_{k} + \delta_{jk}) \, d\zeta = A_{mk}^{\text{eff}}, \quad m \neq 1.
\end{align*}
Then the weak formulation of the limit problem \eqref{limit_rescaled_weak form} becomes
 \begin{align*}
&\int_{\mathbb{R}^{d}} A_{11}^{\text{eff}} \partial_{z_{1}}^{\mu^{*}} w(z_{1},0) \partial_{z_1} \phi(z_{1},0) \, d\mu^{*} + \frac{\mu''(0)}{2} \langle \rho_{\Psi}(0, \cdot) \rangle \int_{\mathbb{R}^{d}} |z_{1}|^{2} w(z_{1},0) \phi(z_{1},0) \, d\mu^{*} \nonumber\\
&~~~+ \int_{\mathbb{R}^{d}} c^{\text{eff}} w(z_{1},0) \phi(z_{1},0) \, d\mu^{*} = \nu \langle \rho_{\Psi}(0, \cdot) \rangle \int_{\mathbb{R}^{d}} w(z_{1},0) \phi(z_{1}, 0) \, d\mu^{*}.
 \end{align*}
Choosing $N_{i}$ as a test function in the cell problem of $N_{k}$ we have
\begin{align*}
    &\int_{Y} a_{\Psi}(0,\zeta) \nabla_{\zeta} ( N_{k} + \zeta_{k}) \cdot \nabla_{\zeta} N_{i} \, d\zeta = 0,
\end{align*}
then $A_{ik}^{\text{eff}}$ can be written as 
\begin{align*}
    A_{ik}^{\text{eff}} = \int_{Y} a_{\Psi}(0, \zeta) \nabla_{\zeta}(N_{k} + \zeta_{k}) \cdot \nabla_{\zeta}(N_{i} + \zeta_{i}) \, d\zeta
\end{align*}
which gives
\begin{align*}
A_{11}^{\text{eff}} = A^{\text{eff}} e_{1} \cdot e_{1}
& = \frac{1}{|Y|} \int_{Y} a(0, \zeta) \nabla(\zeta_{1} + N_{1}(\zeta)) \cdot \nabla(\zeta_{1} + N_{1}(\zeta)) \, d\zeta\\ 
&\ge \frac{\Lambda}{|Y|} \int_{Y} | \nabla(\zeta_{1} + N_{1}(\zeta)) |^{2} \, d\zeta.
\end{align*}
Assuming that $\partial_{\zeta_{i}}(\zeta_{1} + N_{1}(\zeta)) = 0$, for all $i$ leads to the contradiction since $N_{1}$ is periodic in $\zeta_{1}$. Thus, the effective coefficient is strictly positive.

Denoting $a^{\text{eff}} := A_{11}^{\text{eff}}$, $w(z_{1}) := w(z_{1}, 0)$, the last integral identity is  the weak formulation of the harmonic oscillator equation on $\mathbb R$:
\begin{align}
\label{limit_rescled_prob}
    &-a^{\text{eff}} w'' + \left( \frac{\mu''(0)}{2} \langle \rho_{\Psi}(0, \cdot) \rangle |z_{1}|^{2} + c^{\text{eff}} \right) w
    = \nu \langle \rho_{\Psi}(0, \cdot) \rangle w
,\quad w \in L^{2}(\mathbb{R}).
\end{align}
Due to the normalization condition and the strong convergence of $w^{\ve}$ in $L^{2}(\mathbb{R}^{d}, \mu_{\ve})$ (see Lemma \ref{lm:compactness}), the limit function $w(z_{1}) \neq 0$. Thus, $(\nu, w(z_{1}))$ is an eigenpair of the effective spectral problem \eqref{limit_rescled_prob}.
\begin{remark}
    The eigenpairs $(\nu_{j}, w_{j})$ of the Sturm-Liouville problem
    \begin{align*}
        -a^{\text{eff}} w'' + \left( \frac{\mu''(0)}{2} \langle \rho
        _{\Psi}(0, \cdot)\rangle |z_{1}|^{2} + c^{\text{eff}} \right) w = \nu \langle \rho_{\Psi}(0, \cdot) \rangle w, \quad w \in L^{2}(\mathbb{R})
    \end{align*}
    admit explicit representation: 
    \begin{align*}
    &\nu_{j}  = \big(c^{\text{eff}} + (2j - 1) \sqrt{\frac{a^{\text{eff}} \mu''(0)}{2}}\big)/\langle \rho_{\Psi}(0, \cdot) \rangle,\\
    &
    w_{j}(z_{1}) = {H_{j}(\theta^{1/4} z_{1})} e^{-\sqrt{\theta}z_{1}^{2}/2}, \,\, j = 1,2, \ldots,
    \end{align*}
    where $\displaystyle\theta = \frac{\mu''(0)}{2a^{\text{eff}}}$ and $\displaystyle H_{j}(x) = e^{x^{2}} \frac{d^{j-1}}{dx^{j-1}} e^{-x^{2}}$ are the Hermite polynomials. Note that the eigenvalues $\nu_j$ are simple \cite{Bohm1986}.
\end{remark}
\section{Convergence of spectra}
\label{sec convergence of spectra}

In this section we will show that the $j$th eigenvalue of the rescaled problem \eqref{eq:rescle-final} converges to the $j$th eigenvalue of the homogenized problem \eqref{limit_rescled_prob}, as well as the convergence of the corresponding eigenfunctions. 
\begin{lemma}
    For sufficiently small $\ve$, along a subsequence, the eigenvalues $\ve \nu_j^\ve$ of problem \eqref{eq:rescle-final} are simple.
\end{lemma}
\begin{proof}
{Suppose that some eigenvalue $\ve \nu^{\ve}$ of \eqref{eq:rescle-final} has multiplicity at least two for all $\ve>0$, i.e. there exist two linearly independent eigenfunctions $v_{1}^{\ve}$, $v_{2}^{\ve}$ corresponding to $\ve \nu^{\ve}$. Let us normalize the eigenfunctions by $\int_{\mathbb R^d} v_1^\ve v_2^\ve\, d\mu_\ve = \delta_{ij}$. By the strong $L^2$-compactness, $v_{1}^{\ve}$ and $v_{2}^{\ve}$, extended by zero to $\mathbb R \times \ve^{1/2} Q$, converge strongly in $L^{2}(\mathbb{R}^{d}, \mu_\ve)$ to the eigenfunctions $v_1, v_2$ of the limit problem corresponding to some eigenvalue $\nu^{*}$. Passing to the limit in the normalization condition yields $\int_{\mathbb R} v_1\, v_2\, dx =0$.
Since the eigenvalues of the one-dimensional limit problem are simple, $v_{1}$, $v_{2}$ should be linearly dependent, which leads to a contradiction.
}
\end{proof}

The next lemma shows that the convergence of eigenvalues, as $\ve \to 0$, preserves the order.
\begin{lemma}
   For any $j$, the $j$th eigenvalue $\nu_{j}^{\ve}$ of problem \eqref{eq:rescle-final} converges to the $j$th eigenvalue $\nu_{j}$ of \eqref{limit_rescled_prob}, and the corresponding eigenfunction $w_j^\ve$ converges, along a subsequence, to the eigenfunction $w_j$ of \eqref{limit_rescled_prob}.  
\end{lemma}
\begin{proof}
    By the a priori estimates and the compactness, we have proved that all the eigenvalues of \eqref{eq:rescle-final} converge to some eigenvalues of \eqref{limit_rescled_prob}. It is left to prove that all eigenvalues of the effective problem \eqref{limit_rescled_prob} are limits of some eigenvalues of \eqref{eq:rescle-final}. 
    \noindent
    We follow the ideas of the proof of Lemma 3.12 in \cite{PaPe-ApplicAnal-2014}  and will prove this by contradiction. Assume that the first eigenvalue $\nu_{1}^{\ve}$ of \eqref{eq:rescle-final}, converge to the second eigenvalue $\nu_{2}$ of \eqref{limit_rescled_prob}, and not to $\nu_1$. The first eigenvalue $\nu_{1}^{\ve}$ is simple and the corresponding eigenfunction $v_{1}^{\ve}$ converge to the eigenfunction $v_{2}$.
    By the variational principle,
    \begin{align*}
        \nu_{1}^{\ve} = \inf_{w} \frac{\mathcal{F}_{\ve}(w)}{\int_{\mathbb{R}^{d}} \rho_{\Psi}^{\ve} w^{2}\, d\mu_{\ve}},
    \end{align*}
    where the infimum is taken over $H^{1}(\mathbb{R}^{d}, \mu_{\ve}) \setminus \{0\}$ such that $w\Big|_{z_{1} = \ve^{-1/2} \Gamma_{\ve}^{\pm}} = 0$, and
    \begin{align*}
        \mathcal{F}_{\ve}(w) :=
      &\int_{\mathbb{R}^{d}} a_{\Psi}^{\ve}(\sqrt{\ve}z) \nabla w \cdot \nabla w \, d\mu_{\ve}  + \int_{\mathbb{R}^{d}}  (a^{\ve} \nabla_{z} \Psi^{\ve} \cdot \nabla(\Psi^{\ve} w^{2} ) \, d\mu_{\ve}\\
      &- \int_{\mathbb{R}^{d}} \Psi^{\ve} \mathrm{div}_{z} (a^{\ve} \nabla_{\zeta} \Psi^{\ve} ) w^{2} \, d\mu_{\ve}\\
    &+ \int_{\mathbb{R}^{d}} \frac{\mu(\sqrt{\ve}z_{1}) - \mu(0)}{\ve} \rho_{\Psi}^{\ve}(\sqrt{\ve}z) w^{2} \, d\mu_{\ve}.
  \end{align*}
    The minimum is attained on the first eigenfunction $w_{1}^{\ve}$. Since $\nu_{1}^{\ve} \rightarrow \nu_{2}$, as $\ve \rightarrow 0$, we can write $\nu_{1}^{\ve} = \nu_{2} + o(1)$, as $\ve \rightarrow 0$.

    Let $w_{1}(z_{1})$ be the first eigenfunction of \eqref{limit_rescled_prob} and $N$ be the normalized solution of the auxiliary cell problem {satisfied by $N_{1}$}. Denote
    \begin{align*}
        W_{\ve} = \left( w_{1}(z_{1}) + \ve^{1/2} N \left( \frac{z}{\sqrt{\ve}} \right) w_{1}'(z_{1}) \right) \phi_{\ve}(z_{1}),
    \end{align*}
    where
    \begin{align*}
        \phi_{\ve}(z_{1}) = \begin{cases}
            &1, ~~ z_{1} \in [-\frac{\ve^{-1/2}}{6} , \frac{\ve^{-1/2}}{6}]\\
            &0, ~~ z_{1} \in \mathbb{R} \setminus [-\frac{\ve^{-1/2}}{3} , \frac{\ve^{-1/2}}{3}]
        \end{cases}
    \end{align*}
    and such that $0 \le \phi_{\ve} \le 1$, $| \phi_{\ve}'(z_{1}) | \le C \ve^{1/2}$.\\
The cut-off function is introduced in order to make the test function $W_{\ve}$ satisfy the Dirichlet boundary condition at the ends of the rod. Using Corollary \eqref{corollary_two scale_average} and taking into account the exponential decay of $w_{1}(z_{1})$, we estimate the norm of $W_\ve$:
    \begin{align*}
        \|W_{\ve}\|^{2}_{L^{2}(\mathbb{R}^{d}, \mu_{\ve})} &= \int_{\mathbb{R}^{d}} \Big( w_{1}(z_{1}) + \ve^{1/2} N \Big( \frac{z}{\sqrt{\ve}} \Big) w_{1}'(z_{1}) \Big)^{2} |\phi_{\ve}(z_{1})|^{2} \, d\mu_{\ve}\\
        &= \int_{\mathbb{R}^{d}} |w_{1}(z_{1})|^{2} |\phi_{\ve}(z_{1})|^{2} \, d\mu_{\ve} + \ve \int_{\mathbb{R}^{d}} \left|N \left(\frac{z}{\sqrt{\ve}} \right) w_{1}'(z_{1}) \right|^{2} |\phi_{\ve}(z_{1})|^{2} \, d\mu_{\ve}\\
        &~~~~+ 2\sqrt{\ve} \int_{\mathbb{R}^{d}} w_{1}(z_{1}) N \left(\frac{z}{\sqrt{\ve}} \right) w_{1}'(z_{1}) |\phi_{\ve}(z_{1})|^{2} \, d\mu_{\ve} \\
        &= \int_{\mathbb{R}^{d}} |w_{1}(z_{1})|^{2} |\phi_{\ve}(z_{1})|^{2} \, d\mu_{\ve} + o(\ve)\\
        &= \int_{\mathbb{R}} |w_{1}(z_{1})|^{2} \, dz_{1} + o(\sqrt{\ve}), ~\text{as}~ \ve \rightarrow 0.
    \end{align*}
    \noindent
    Let us compute the derivatives of $W_\ve$:
    \begin{align*}
        &\partial_{z_{i}} W_{\ve} = \left( \delta_{1i} + \partial_{\zeta_{i}} N(\zeta)\Big|_{\zeta = z/\sqrt{\ve}} \right) w_{1}'(z_{1}) \phi_{\ve}(z_{1}), ~~i \neq 1,\\
        &\partial_{z_{1}} W_{\ve} = \left( \delta_{11} + \partial_{\zeta_{1}} N(\zeta)\Big|_{\zeta = z/\sqrt{\ve}} \right) w_{1}'(z_{1}) \phi_{\ve}(z_{1}) + \ve^{1/2} N \left(\frac{z}{\sqrt{\ve}} \right) w_{1}''(z_{1}) \phi_{\ve}(z_{1}) \\
        &~~~~~~~~~+ \left( w_{1}(z_{1}) + \ve^{1/2} N \left(\frac{z}{\ve} \right) w_{1}'(z_{1}) \right) \phi_{\ve}'(z_{1}).
    \end{align*}
    Sustituting $W_{\ve}$ into the functional $\mathcal{F}_{\ve}$, we obtain
    \begin{align*}
        \mathcal{F}_{\ve}(W_{\ve}) &= \int_{\mathbb{R}^{d}} (a_{\Psi}^{\ve})_{ij} \left( \delta_{1j} + \partial_{\zeta_{j}} N \left(\frac{z}{\sqrt{\ve}} \right) \right) \left( \delta_{1i} + \partial_{\zeta_{i}} N \left(\frac{z}{\sqrt{\ve}} \right) \right) \phi_{\ve}^{2}(z_{1}) (w_{1}'(z_{1}))^{2} \, d\mu_{\ve}\\
        &~~+ \int_{\mathbb{R}^{d}} (a^{\ve} \nabla_{z} \Psi^{\ve}) \cdot \nabla(\Psi^{\ve} W_{\ve}^{2}) \, d\mu_{\ve} + \frac{1}{\sqrt{\ve}} \int_{\mathbb{R}^{d}} \Psi^{\ve} \mathrm{div}_{z} (a^{\ve} \nabla_{\zeta} \Psi^{\ve}) W_{\ve}^{2} \, d\mu_{\ve}\\
        &~~+ \int_{\mathbb{R}^{d}} \frac{\mu(\sqrt{\ve}z_{1}) - \mu(0)}{\ve} \rho_{\Psi}^{\ve} W_{\ve}^{2} \, d\mu
        _{\ve} + o(1), ~~\ve \rightarrow 0.
    \end{align*}
    Taking the limit as $\ve \rightarrow 0$
    \begin{align*}
        \mathcal{F}_{\ve}(W_{\ve}) &= \int_{\mathbb{R} \times Y} (a_{\Psi}(0, \zeta))_{ij} ( \delta_{1j} + \partial_{\zeta_{j}} N(\zeta)) ( \delta_{1i} + \partial_{\zeta_{i}} N(\zeta ) (w_{1}'(z_{1}))^{2} \, d\zeta dz_{1}\\
        &+ \int_{\mathbb{R} \times Y} (a \nabla_{x} \Psi)(0, \zeta) \cdot \nabla_{\zeta} (\Psi(0, \zeta)) (w_{1}(z_{1}))^{2} \, d\zeta dz_{1}\\
        &- \int_{\mathbb{R} \times Y} \Psi(0, \zeta) \mathrm{div}(a \nabla_{\zeta} \Psi)(0, \zeta) (w_{1}(z_{1}))^{2} \, d\zeta dz_{1}\\
        &+ \frac{\mu''(0)}{2} \langle \rho_{\Psi}(0, \cdot) \rangle \int_{\mathbb{R}} |z_{1}|^{2}( w_{1}(z_{1}))^{2} \, dz_{1} + o(1)\\
        &= \int_{\mathbb{R}} a^{\text{eff}} (w_{1}(z_{1}))^{2} \, dz_{1} + \int_{\mathbb{R}} c^{\text{eff}} (w_{1}(z_{1}))^{2} \, dz_{1}\\
        &+ \frac{\mu''(0)}{2} \langle \rho_{\Psi}(0, \cdot) \rangle \int_{\mathbb{R}} |z_{1}|^{2}( w_{1}(z_{1}))^{2} \, dz_{1} + o(1), ~\ve \to 0.
    \end{align*}
    Therefore,
    \begin{align*}
       &\frac{ \mathcal{F}(W_{\ve})}{\int_{\mathbb{R}^{d}}\rho_{\Psi}^\ve W_{\ve}^{2} \, d\mu_{\ve}}\\
       &= \frac{\int_{\mathbb{R}} a^{\text{eff}} (w_{1}')^{2} \, dz_{1} + \int_{\mathbb{R}} c^{\text{eff}} (w_{1})^{2} \, dz_{1} + \frac{\mu''(0)}{2} \langle \rho_{\Psi}(0, \cdot) \rangle \int_{\mathbb{R}} |z_{1}|^{2}( w_{1})^{2} \, dz_{1}} { \langle 
\rho_{\Psi}(0, \cdot) \rangle \int_{\mathbb{R}} (w_{1})^{2} \, dz_{1}}+ o(1)\\
&= \nu_{1} + o(1), ~\ve \to 0.
    \end{align*}
Since $\nu_{1}<\nu_2$, we have constructed a test function giving a smaller eigenvalue than $\nu_{1}^{\ve} = \nu_{2} + o(1)$, which is a contradiction.
\end{proof}

\appendix
\section{Two-scale convergence in spaces with measure}
\label{Appendix A}
\begin{definition}
Let $\mu_\ve$ be the measure defined by \eqref{def:mu_eps}.
A sequence $g^\ve(x) \in L^{2}(\mathbb{R}^{d}, \mu_{\ve})$ is said to converge weakly in $L^{2}(\mathbb{R}^{d}, \mu_{\ve})$ to a function $g(x_{1}) \in L^{2}(\mathbb{R}^{d}, \mu^{*})$, as $\ve \rightarrow 0$, if
\begin{itemize}
\item[(i)] 
$\|g^{\ve}\|_{L^{2}(\mathbb{R}^{d}, \mu_{\ve})} \le C$,

\item[(ii)] for any $\phi \in C_{c}^{\infty}(\mathbb{R}^{d})$ the following relation holds:
\begin{align*}
    \lim_{\ve \rightarrow 0} \int_{\mathbb{R}^{d}} g^{\ve}(x) \phi(x) \, d\mu_{\ve} = \int_{\mathbb{R}^{d}} g(x_{1}) \phi(x) \, d\mu^{*}.
\end{align*}
\end{itemize}
A sequence $g^{\ve}$ is said to converge strongly to $g(x_{1})$ in $L^{2}(\mathbb{R}^{d}, \mu_{\ve})$, as $\ve \rightarrow 0$, if it converges weakly and 
\begin{align*}
    \lim_{\ve \rightarrow 0} \int_{\mathbb{R}^{d}} g^{\ve}(x) \psi^{\ve}(x) \, d\mu_{\ve} = \int_{\mathbb{R}^{d}} g(x_{1}) \psi(x_{1}) \, d\mu^{*}
\end{align*}
for any sequence $\{\psi^{\ve}(x)\}$ weakly converging to $\psi(x_{1})$ in $L^{2}(\mathbb{R}^{d}, \mu_{\ve})$.
\end{definition}
The weak compactness results are valid in spaces with measure, we refer to \cite{ZikovTwoscale}.\\

In the present context, two-scale convergence is described as follows.

\begin{definition}
    We say $g^{\ve} \in L^{2}(\mathbb{R}^{d}, \mu_{\ve})$ converges two-scale weakly, as $\ve \rightarrow 0$, in $L^{2}(\mathbb{R}^{d}, \mu_{\ve})$ if
    \begin{itemize}
    \item[(i)]
    $\|g^{\ve}\|_{L^{2}(\mathbb{R}^{d}, \mu_{\ve})} \le C, ~~~ \ve >0$,

    \item[(ii)] there exists a function $g(x_{1},y) \in L^{2}(\mathbb{R}^{d} \times Y, \mu^{*} \times d\zeta)$ such the following limit relation holds:
    \begin{align*}
        \lim_{\ve \rightarrow 0} \int_{\mathbb{R}^{d}} g^{\ve}(x) \phi(x) \psi(\frac{x}{\sqrt{\ve}}) \, d\mu_{\ve} = \frac{1}{|Y|} \int_{\mathbb{R}^{d}} \int_{Y} g(x_{1},y) \phi(x_{1}, 0) \psi(\zeta) \, d\zeta \, d\mu^{*},
    \end{align*}
    for any $\phi \in C_{c}^{\infty}(\mathbb{R}^{d})$ and $\psi(\zeta) \in C^{\infty}(\Bar{Y})$ periodic in $\zeta_{1}$.
    \end{itemize}
\end{definition}
We write $g^{\ve} \overset{2}{\rightharpoonup} g(x_{1},y)$
if $g^{\ve}$ converges two-scale weakly to $g(x_{1},y)$ in $L^{2}(\mathbb{R}^{d}, \mu_{\ve})$.\\

Note that the last definition holds for more general classes of test functions, e.g. $\Phi(x,y) \in C(\mathbb{R}^{d}; L^{\infty}(Y))$ or $\Phi(x,y) = \phi(x) \psi(y)$ with $\phi \in C(\mathbb{R}^{d})$, $\psi \in L^{2}(Y)$, \cite{ZikovTwoscale}.\\
\begin{lemma}
   Suppose that $g^{\ve}$ satisfies the following estimate
   $$\|g^{\ve}\|_{L^{2}(\mathbb{R}^{d}, \mu_{\ve})} \le C.$$
Then $g^{\ve}$, up to a subsequence, converges two-scale weakly in $L^{2}(\mathbb{R}^{d}, \mu_{\ve})$ to some function $g(x_{1},y) \in L^{2}(\mathbb{R}^{d} \times Y, \mu^{*} \times d\zeta)$.   
\end{lemma}

\begin{definition}
A sequence $g^{\ve}$ is said to converge two-scale strongly to a function $g(x_{1},y) \in L^{2}(\mathbb{R}^{d} \times Y, \mu^{*} \times d\zeta)$ if
\begin{itemize}
\item[(i)] $g^{\ve} \overset{2}{\rightharpoonup} g(x_{1},y)$.
\item[(ii)] The following limit relation holds
\begin{align*}
    \lim_{\ve \rightarrow 0} \int_{\mathbb{R}^{d}} (g^{\ve}(x))^{2} \, d\mu_{\ve} = \frac{1}{|Y|} \int_{\mathbb{R}^{d}} \int_{Y} (g(x_{1},y))^{2} \, d\zeta \, d\mu^{*}
\end{align*}
\end{itemize}
\end{definition}
We write $g^{\ve} \overset{2}{\rightarrow} g(x_{1},y)$
if $g^{\ve}$ converges two-scale strongly to the function $g(x_{1},y)$ in $L^{2}(\mathbb{R}^{d}, \mu_{\ve})$.

\begin{lemma}
\label{lm:compactness}
    Assume that $v_\ve$ is such that
    \begin{align*}
        \int_{\mathbb R^d} |\nabla v|^2\, d\mu_\ve
        + \int_{\mathbb R^d} |x_1 v|^2\, d\mu_\ve \le C.
    \end{align*}
    Then $v_\ve$ converges strongly in $L^2(\mathbb R^d, \mu_\ve)$ to $v\in L^2(\mathbb R)$.
\end{lemma}
{The proof of Lemma \ref{lm:compactness} follows the lines of Lemma 4.4 in \cite{piat2013localization}.}

\section{Integral estimates for oscillating functions}
\label{Appendix B}
The proof of the following integral estimate for oscillating functions can be found in \cite{PaPe-ApplicAnal-2014}.
\begin{lemma}
\label{lm:oscillating-integrals}
    Let $v_{\ve} \in H^{1}_{0}(\Omega_{\ve})$, $v_{\ve}(\pm 1, x') = 0$, and $w(x_{1},y) \in C^{1}(\bar{I}, L^{\infty}(Y))$. Then as $\ve \rightarrow 0$,
    \begin{align*}
        &\int_{\Omega_{\ve}} w(x_{1}, \frac{x}{\ve}) v_{\ve}^{2}(x) \, dx - \frac{1}{|Y|}\int_{\Omega_{\ve}} (\int_{Y} w(x_{1},y) \, dy)\, v_{\ve}^{2}(x) \, dx\\ \vspace{1mm}
        &~~= O(\ve\|v_{\ve}\|_{L^{2}(\Omega_{\ve})} \|\nabla v_{\ve}\|_{L^{2}(\Omega_{\ve})}).
    \end{align*}
\end{lemma}

\begin{corollary}
\label{corollary_two scale_average}
    Let $w_{\ve} \in H^{1}(\mathbb{R}^{d}, \mu_{\ve})$ be such that $w_{\ve} = 0$ on $\ve^{-1/2} \Gamma_{\ve}^{\pm}$, and $c(x_{1},y) \in C^{1}(\mathbb{R}; L^{\infty}(Y)$. Then as $\ve \rightarrow 0$, 
    \begin{align*}
     &\int_{\mathbb{R}^{d}} c(\sqrt{\ve} z_{1}, \frac{z}{\sqrt{\ve}}) w_{\ve}^{2}(z) \, d\mu_{\ve} - \int_{\mathbb{R}^{d}} \left( \int_{Y} c(\sqrt{\ve}z_{1}, \zeta) \, d\zeta \right) w_{\ve}^{2}(z) \, d\mu_{\ve}\\
     &~~~= O\left( \sqrt{\ve} \|w_{\ve}\|_{L^{2}(\mathbb{R}^{d}, \mu_{\ve})} \| \nabla w_{\ve} \|_{L^{2}(\mathbb{R}^{d}, \mu_{\ve})} \right).
    \end{align*}
\end{corollary}

\section*{Acknowledgments}
I.P. would like to thank Yves Capdeboscq for discussions about the regularity of the solution of an auxiliary periodic spectral problem.

\bibliographystyle{plain}
\bibliography{refs}

\end{document}